\documentclass{amsart}

\usepackage{amsmath,amsxtra,amssymb,amsfonts,latexsym, mathrsfs, dsfont,bm,enumitem}
\usepackage{bbm}
\usepackage[usenames]{color}
\usepackage[hyperindex]{hyperref}
\usepackage[initials]{amsrefs}
\usepackage{tikz}
\usepackage{comment}
\usepackage{xcolor}
\usepackage{graphics,epsf} 
\usepackage{pgfplots}
\usepackage{graphicx}
\usepackage{wrapfig}        
\usepackage{caption}

\newtheorem{theorem}{Theorem}[section]
\newtheorem{lemma}[theorem]{Lemma}
\newtheorem{proposition}[theorem]{Proposition}

\newtheorem*{vdcthm*}{Theorem}

\makeatletter
\def\namedlabel#1#2{\begingroup
	#2%
	\def\@currentlabel{#2}%
	\phantomsection\label{#1}\endgroup
}
\makeatother

\renewcommand{\q}{\quad}

\newcommand{\cn}{\mathcal N}

\newcommand{\bn}{\mathbf{n}}

\DeclareMathOperator\supp{supp}

\newcommand{\N}{\mathbb{N}}

\newcommand{\Q}{\mathbb{Q}}
\newcommand{\R}{\mathbb{R}}
\newcommand{\C}{\mathbb{C}}

\newcommand{\lf}{\lfloor}
\newcommand{\rf}{\rfloor}
\newcommand{\varep}{\varepsilon}

\newcommand{\aln}[1]{\begin{align}#1\end{align}}
\newcommand{\alnn}[1]{\begin{align*}#1\end{align*}}
\newcommand{\eqn}[1]{\begin{equation}#1\end{equation}}
\newcommand{\eqnn}[1]{\begin{equation*}#1\end{equation*}}
 
\newcommand{\bee}{\mathbf{e}}
\newcommand{\bt}{\mathbf{t}}
\newcommand{\bz}{\mathbf{0}}
\newcommand{\bone}{\mathbf{1}}
\newcommand{\btwo}{\mathbf{2}}

\newcommand{\lddl}{\lambda\frac{d}{d\lambda}}

\begin{document}
	\title[Oscillatory Integrals via Real Analysis]{Some Oscillatory Integral Estimates Via Real Analysis}
	
	\author{Maxim Gilula}
	\address{Department of Mathematics, Michigan State University, East Lansing, MI, 48824, USA}
	\email{gilulama@math.msu.edu }
	
	\begin{abstract}
		We study oscillatory integrals in several variables with analytic, smooth, or $C^k$ phases satisfying a nondegeneracy condition attributed to Varchenko. With only real analytic methods, Varchenko's estimates are rediscovered and generalized. The same methods are pushed further to obtain full asymptotic expansions of such integrals with analytic and smooth phases, and finite expansions with error assuming the phase is only $C^k$. The Newton polyhedron appears naturally in the estimates; in particular, we show precisely how the exponents appearing in the asymptotic expansions depend only on the geometry of the Newton polyhedron of the phase. All estimates proven hold for oscillatory parameter real and nonzero, not just asymptotically.
	\end{abstract}
	\maketitle
	
	\section{Introduction}
	The purpose of this paper is to develop a real analytic method of studying oscillatory integrals of the form
	\eqn{I(\lambda)=\int_{\R^d}e^{i\lambda\phi(x)}\psi(x)dx,\label{oscint}}
	where $\lambda$ is a real parameter, $\phi:\R^d\to\R$ is real analytic, smooth, or $C^k$, satisfying a certain nondegeneracy condition, and $\psi$ is a smooth or $C^k$ cutoff supported close enough to the origin. While developing this method, we reprove and improve some classical results, and obtain some new estimates. For example, an expansion of $I(\lambda)$ for certain nondegenerate $C^k$ phases is developed for all $\lambda\neq 0.$
		
	Van der Corput's Lemma completely characterizes the behavior of $I(\lambda)$ when $\phi:(a,b)\to\R$ is $C^k$ (or, if we have a cutoff supported on $(a,b)$):
	
	\begin{vdcthm*}[van der Corput\cite{vandercorput35}] 
		Let $\phi:(a,b)\to\R$ be $C^k$ and assume $|\phi^{(k)}(x)|\ge 1$ for all $x\in (a,b).$ Then
		\eqnn{\Bigg|\int_a^b e^{i\lambda\phi(x) dx}\Bigg|\le c_k |\lambda|^{-1/k}}
		holds for all $\lambda\neq 0$ for $k\ge 2$. It also holds for $k=1,$ assuming $\phi'$ is monotone. 
	\end{vdcthm*}
	Van der Corput's lemma has been a key tool for finding decay rates of solutions to  differential equations (e.g., Bessel functions). The key steps of the proof (for example, in Stein\cite{s93}) are as follows: a lower bound is obtained away from the singularity of the phase, the method of stationary phase is applied, then optimization is used to obtain the best bound in $\lambda$. Van der Corput-type lemmas are crucial for understanding measures supported on surfaces, along with many other important and current research questions, but oscillatory integrals are not well understood in high dimensions. Many authors have given their take to the vital question \lq\lq What is van der Corput's lemma in higher dimensions?\rq\rq  Just a handful of significant papers attempting to answer this include Carbery-Christ-Wright\cite{ccw99}, Phong-Stein-Sturm\cite{pss01}, Carbery-Wright\cite{cw02} (a paper titled precisely this question), and too many more to name. Because of the difficulties posed by singularities in higher dimensions, this question has not yet been fully answered. In most answers to this question, sharpness is traded for uniformity (multiple of the beautiful results in the above mentioned papers), but in some cases uniformity is traded for sharpness, e.g., Varchenko\cite{varchenko76}: under a nondegeneracy condition on the real analytic phase $\phi$, he showed that for smooth $\psi$ supported close enough to the origin, 
	\eqnn{\Bigg{|}\int_{\R^d} e^{i\lambda\phi(x)}\psi(x)dx\Bigg{|}= O\Big(\lambda^{-1/t}\log^{d-1-k}(\lambda)\Big)}
	as $\lambda\to\infty,$ where $t$ and $k$ can be read directly from the Newton polyhedron of $\phi.$ Moreover, this bound is sharp in both exponents if $t>1$ and $\psi(\bz)\neq 0.$
	
	Many papers obtaining sharp estimates, such as Varchenko's, and more recently Kamimoto-Nose\cite{kn16}, borrow algebraic techniques mainly because of the difficulty caused by the singularities of the phase. Examples of such techniques involve adapted coordinates resolution of singularities, toric varieties, and finding poles of Zeta functions adapted to these problems. However, the original proof is purely analytic, and it seems unlikely that these algebraic methods are the right setting for answering questions of uniformity. It should be noted that these methods do emphasize on the main difficulties of estimating such integrals: singularities of analytic functions are not well understood, and require powerful tools.
	
	The main ideas in this paper are organized as follows: a lower bound away from the singularities of the phase is obtained, integration by parts away from the singularities is used to obtain a quantitative stationary phase result, then optimization is applied to obtain the optimal bound for reproving Varchenko's upper bound. Since these ideas so closely resemble the original proof of van der Corput, they seem to provide a natural setting for answering what van der Corput is in higher dimensions. Dyadic decomposition is the bread and butter of many proofs in harmonic analysis and is natural in this setting because understanding $I(\lambda)$ requires a delicate analysis of cancellations near the singularities of $\phi$ (see Rychkov\cite{rychkov01} for just one great example of such an argument for smooth functions in two variables). Such decompositions into more general shapes has been a very successful technique for finding multilinear estimates, e.g., in Phong-Stein\cite{ps94} and Phong-Stein-Sturm\cite{pss01}. The asymptotic expansion for nondegenerate real analytic and smooth phases has also been recently studied by Cho-Kamimoto-Nose\cite{ckn13} with mainly algebraic methods. We provide evidence for the power of our methods via a new proof of the sharp bound originally proven in Varchenko's seminal work for nondegenerate analytic phases\cite{varchenko76}. We also prove Varchenko's upper bound for oscillatory integrals with smooth and $C^k$ phases under a similar nondegeneracy condition. Then, a full asymptotic expansion for nondegenerate smooth and analytic phases is developed, as well as an asymptotic expansion with finitely many terms for nondegenerate $C^k$ phases. The proof here also mirrors the standard proofs of the asymptotic expansion for one-dimensional phases: solving a differential inequality involving $I(\lambda)$.  One more feature of van der Corput's lemma is that the estimates hold for all $\lambda\neq 0,$ but Varchenko only has an estimate asymptotic in $\lambda.$ Gressman\cite{pgs16} and others have noted that bounds for all $\lambda$ are important for obtaining stability results. Our results hold for all $\lambda\neq 0,$ including the asymptotic expansion with finitely many terms (though we still refer to it as an asymptotic expansion). The main purpose of developing these tools is to apply them to the study of stability questions.
	
	In higher dimensions, the Newton polyhedron has proven itself to be a key tool for describing the behavior of oscillatory integrals (e.g., \cite{ckn13, MR3117305, greenblatt10, greenblatt12, MR2297031, pgs16, kn16, kn16toric, ps94, pss01, varchenko76}, and many more). It is the best combinatorial tool known for characterizing decay of real analytic functions. Naturally, we use it below to keep track of which monomials in a Taylor expansion are the largest. It can be argued that the Newton polyhedron was also used in one dimension, as the polyhedron reduces to a ray on the real line, and its boundary answers which monomial contributes most near the critical point. Since nondegeneracy is automatically satisfied in one variable, all results also apply to $C^k$ functions of one variable, generalizing van der Corput in particular.

	To demonstrate how the Newton polyhedron is used as a geometric tool, let us  briefly discuss the asymptotic expansion of $I(\lambda).$ To prove the asymptotic expansion, we consider the decay of $(\lambda \tfrac{d}{d\lambda} +p)I(\lambda)$. After integration by parts, we are able to show that for certain $p$ depending on the amplitude and phase, the decay of this new integral is shown to be strictly better than that of $I(\lambda)$. After proving a sequence of ODE inequalities involving such operators, we are able to show that 
		\begin{wrapfigure}{L}{5cm}
			\centering
			\captionsetup{justification=centering}
			\begin{tikzpicture}
			\draw[fill=lightgray!30,lightgray!30] (0,4)--(0,3)--(2,0)--(3,0)--(3,4);
			\draw[<->] (0,4)--(0,0)--(3,0);
			\draw[thick, red] (0,3) -- (2,0);
			\foreach \x in {1,...,3}{
				\foreach \y in {1,...,4}{
					\pgfmathsetmacro{\q}{\x/2+\y/3}
					\draw[fill] (\x,\y) circle [radius=.03];
					\draw[blue,->] (0,0) -- (\x,\y);
					\draw[blue, fill] (\x/\q,\y/\q) circle [radius=.05];
				}
			}
			\end{tikzpicture}
			\caption{Newton polyhedron and scalings of Thm \ref{thm2}.}
		\end{wrapfigure}
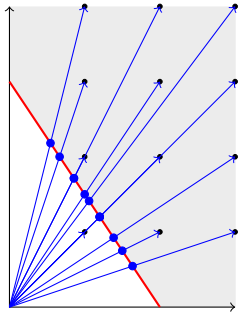
	\eqn{\Bigg{|}I(\lambda)- \sum_{j=0}^{n-1}\sum_{r=0}^{d_j-1}a_{j,r}(\psi)|\lambda|^{-p_j}\log^{d_j-1-r}(|\lambda|+2)\Bigg{|}\le C|\lambda|^{-p_{n}}\log^{d_n-1}(|\lambda|+2),\label{expansionI}}
	 where $C$ is some constant independent of $\lambda$ but depending on $\phi$, $\psi,$ and their higher derivatives. For nondegenerate analytic and smooth phases, $n$ can be any natural number. Nondegenerate $C^k$ phases have a bound on $n$ depending on $k$ as well as the geometry of the Newton polyhedron of $\phi$. The precise statement of the above is Theorem \ref{thm2}. The expansion is completely characterized by exponents $p_j$ and $d_j$ read directly off from the boundary of the Newton polyhedron by considering scalings of positive integer lattice points (see Figure 1). For example, we can scale the \lq\lq smallest\rq\rq positive integer lattice point $\bone=(1,\dotsc,1)$ by $1/t$, where $\bt:=t\bone$ lies on the boundary of the Newton polyhedron. It is no coincidence that $-1/t$ is the largest exponent in \eqref{expansionI}, and that $d_0=\min\{d,n\}$ where $n$ is the number of codimension $1$ faces containing $\bt$.
	 
	 In order to prove the lower bound away from singularities (Lemma \ref{lem1}), we delicately use nondegeneracy: there is a tug of war between the slowest decaying monomials appearing in the Taylor expansion of $\phi$ and the monomials that are nearly slowest decaying. The crucial step is making sure there is not too much cancellation between monomials in some sense, which would force the gradient to be too small. We show that the nondegeneracy condition guarantees this does not happen, and we obtain a result similar to that of Lojasiewicz for analytic functions\cite{lojasiewicz59}, except ours is sharpest possible. In fact, we show that the nondegeneracy considered by Varhcenko is equivalent to the sharpest possible lower bound on the growth of the gradient. This is the place others might use resolution of singularities, but we avoid it with elementary analytic methods: triangle inequality and linear algebra are all we need, after developing a good sense of the cancellations of $\nabla\phi$ near its singularities. We use this pointwise estimate to obtain bounds over dyadic boxes. Then $I(\lambda)$ is dyadically decomposed via partition of unity. Since $\phi$ has no singularities in any such box, Lemma \ref{thm1} provides a sharp estimate by relying on the previous lemma (sharp in the sense that the estimates are later used to reprove Varchenko's sharp upper bound on $I(\lambda)$).  We avoid difficult integration by parts by choosing an operator that is most adapted to our integration by parts. To obtain Theorem \ref{wutcor}, from which Varchenko's upper bound follows easily as a special case, we optimize over size and decay estimates guaranteed by stationary phase. Afterwards, we sum over all boxes to obtain the final estimate. So all previous algebraic proofs are reduced in complexity to analysis and linear algebra, in addition to the proof being shorter than others in current literature. The challenge of obtaining sharp exponents usually requires delicate and technical arguments, and this case is no exception. The lower bound, the decomposition, and the optimization are handled carefully so that the cancellation of the integral can still be exploited to obtain sharp exponents, and the most technical arguments are explained with geometric intuition along the way.

	\subsection{Conventions and terminology}
	If $x$ is a $d-$tuple we write $x=(x_1,\dotsc, x_d)$ so that subscripts denote components of a vector. On the other hand, whenever we have a list of $d-$tuples, they are indexed by a superscript, e.g., $\{\alpha^i\}_{1\le i\le d}$. There is one consistent exception: the standard unit normals $\bee_i\in\R^d$ defined componentwise by the Kronecker delta $\bee_{ij}=\delta_{ij}$. We write $\|\cdot \|_p:\R^d\to\R$ for the $\ell^p$ norm on $\R^d$ for $1\le p\le \infty,$ and simply $\|\cdot\|$ for the norm $\|\cdot\|_\infty.$
	
	Let $\R_\ge=[0,\infty).$ In addition to the standard notation for $y\in\R_\ge^d$ and $\alpha\in\R$ that $\partial^\alpha=\frac{\partial^{\alpha_1}}{\partial x_1^{\alpha_1}}\cdots \frac{\partial^{\alpha_d}}{\partial x_d^{\alpha_d}},$ the exponentiation of vectors $y^\alpha=y_1^{\alpha_1}\cdots y_d^{\alpha_d},$ as well as $|y|=y_1+\cdots +y_d$, we make use of some less standard notation for $c\in\R$ and $y, z \in\R_\ge^d.$ The following conventions condense many computations throughout:
	
	\begin{itemize}
		\item $yz = (y_1 z_1,\dotsc, y_d z_d);$
		\item boldface $\mathbf{c}$ denotes the vector $(c,\dotsc, c)$;	
		\item if $c>0,$ denote the vector $(c^{y_1},\dotsc, c^{y_d})$ by $c^y$;
		\item if $c>1$ then $[y, cy]$ is defined to be the box $\prod_{j=1}^d [y_j, cy_j]$.

	\end{itemize}
	For positive real-valued functions $f$ and $g$, we use the notation
	\eqnn{f(x)\lesssim g(x)}
	to express that there is a positive constant $C$ independent of $x$ such that $f(x)\le Cg(x)$ for all $x$ in the common domain of $f$ and $g$. Finally, we define
	\eqnn{I(\lambda)\sim \sum_{j=0}^N a_jE_j(\lambda)}
	to mean that $|I(\lambda)-\sum_{j=0}^n a_jE_j(\lambda)|\lesssim E_{n+1}(\lambda)$ for all $n<N,$ where the implicit constant is independent of $\lambda$ in the stated domain, but may depend on all $a_j$.

	
	
	\section{Main results}
	Let $\phi$ be an analytic function defined in a neighborhood of the origin. Then $\phi$ can be expressed as a uniformly and absolutely convergent series $\phi(x)=\sum_{\alpha}c_\alpha x^\alpha$ in some possibly smaller neighborhood of the origin.  It is assumed throughout that the phase $\phi$ satisfies $\phi(\bz)=0$ and $\nabla\phi(\bz)=\bz$. Define the \textbf{Taylor support} of $\phi$ by $\supp(\phi)=\{\alpha\in\N^d: c_\alpha\neq 0\},$ where we use the convention that $\N$ is the set of nonnegative integers. The \textbf{Newton polyhedron} of $\phi$, denoted $\cn(\phi),$ is defined to be the convex hull of the union
		\eqnn{\bigcup_{\alpha\in\supp(\phi)} \alpha+\R_\ge^d.}
	Given a compact face $F$ of $\cn(\phi),$ define the polynomial
		\eqnn{\phi_F(x)=\sum_{\alpha\in F}c_\alpha x^\alpha.}
	Define $\phi$ to be \textbf{analytic nondegenerate} if $\phi$ is analytic and for all compact faces $F$ of $\cn(\phi)$, $x^\bone\neq 0$ implies $\|x\nabla\phi_F(x)\|\neq 0.$ In other words, for all compact faces $F$ and for all $x$ not contained in any coordinate hyperplane, there is some $1\le j\le d$ such that $x_j\partial_j\phi_F(x)\neq 0.$
	
	Since polyhedra have finitely many extreme points, the Newton polyhedron only requires a finite subset of $\supp(\phi)$ in order to be defined. We use this to motivate the following definitions for $C^k$ functions. If $P$ is a polynomial such that $\cn(P)$ intersects each coordinate axis, we call $P$ \textbf{convenient}. Assume now that $\phi$ is $C^k$ for some $k\ge 1$ in a neighborhood of the origin. Let $P_k$ be the Taylor polynomial of $\phi$ of order at most $k.$ If $P_k$ is convenient, we define $\cn(\phi)=\cn(P_k)$. Finally, we say $\phi=P_k+R_k$ is \boldsymbol{$k-$}\textbf{nondegenerate} if $P_k$ is convenient analytic nondegenerate. Note that if $\phi\in C^m([-1,1])$ is $k-$nondegenerate for some $k\le m$ then $\phi$ is also $m-$nondegenerate.

	In this paper, we reserve the letter $t$ for the \textbf{Newton distance} of the phase under consideration: given $\phi$ such that $\cn(\phi)$ can be defined as above, define the positive real number $t=\inf\{s:\mathbf{s}\in \cn(\phi)\}$ to be the Newton distance of $\phi.$
		
	A crucial step in proving the main theorems is quantifying how $\nabla\phi$ behaves near the origin.
	
	\begin{lemma}
		Assume $\phi$ is analytic nondegenerate or $k-$nondegenerate. For all $\varep\in (0,1)^d$ small enough $($i.e., $\|\varep\|$ small enough$)$, for all $x$ in the box $[\varep, 4\varep],$ and for all $\alpha\in \cn(\phi),$ we have the lower bound 
		\eqnn{\|x\nabla\phi(x)\|\gtrsim \varep^\alpha,}
		where the implicit constant is independent of $\varep.$ \label{lem1}
	\end{lemma}
	\noindent So nondegeneracy implies the sharpest possible growth rate for $\nabla\phi$ around the origin, where \lq\lq small enough\rq\rq is made explicit in (\ref{smallenough}), near the end of the proof of Lemma \ref{lem1}. On the other hand, one can easily see that satisfaction of such a lower bound implies nondegeneracy for analytic functions. Moreover, it implies nondegeneracy for $C^k$ functions under additional assumptions on the remainder term, such as $k-$nondegeneracy. With stationary phase, we prove
	
	\begin{lemma} Let $\beta\in\N^d.$ Let $\phi$ be analytic nondegenerate or $k-$nondegenerate. Assume $\eta:\R^d\to\R$ is $C^k$ with support in $[1,4]^d$. For all $\varep\in(0,1)^d$ small enough, we have the estimate
		\eqnn{\Bigg|\int_{\R^d}e^{i\lambda\phi(x)}x^\beta\eta(\varep_1^{-1}x_1,\dots,\varep_d^{-1}x_d)dx\Bigg|\lesssim \lambda^{-N} \varep^{-(N\alpha-\beta-\bone)}}
		for all $\lambda > 2$, all $0\le N \le k-1$, and all $\alpha \in \cn(\phi),$ where the implicit constant above is independent of $\varep$ and $\lambda.$ If in addition $\phi$ and $\eta$ are smooth then the estimate holds for all $0\le N<\infty.$
		\label{thm1}		
	\end{lemma}
	
	With the help of Lemma \ref{thm1} we prove a quantitative generalization of Varchenko's upper bounds. Below we use the notation $\lf\beta+\bone\rf$, which is explained in the following section. The most important fact to keep in mind is $c=\lf\beta+\bone\rf$ is the constant such that $(\beta+\bone)/c$ is contained in $\partial\cn(\phi).$ For example, Varchenko's estimate is the case $\beta=0:$ the vector $\bone/c\in\partial\cn(\phi)$ if and only if $c=1/t,$ where $t$ is the Newton distance. This lemma may be particularly useful to the readers when $\phi$ satisfies nondegeneracy after a change of coordinates.
	
	\begin{theorem}
		Assume $\phi$ is analytic nondegenerate or $k-$nondegenerate and let $\psi:\R^d\to\R$ be $C^k$ supported close enough to the origin. Assume $\beta\in\N^d$ satisfies $\lf\beta+\bone\rf<k-1$. Let $d_\beta$ be the greatest codimension over all faces of $\cn(\phi)$ containing $(\beta+\bone)/\lf\beta+\bone\rf.$ There is a uniform constant independent of $\lambda>2$ such that
		\eqnn{\Bigg{|}\int_{\R^d}e^{i\lambda\phi(x)}x^\beta\psi(x) dx\Bigg{|}\lesssim \lambda^{-\lf\beta+\bone\rf}\log^{d_\beta-1}(\lambda)}
		If in addition $\phi$ and $\psi$ are smooth then the inequality holds for all $\beta\in\N^d.$
		\label{wutcor}		
	\end{theorem}
	
	For the last theorem, we briefly introduce a well-ordered set $\mathcal{E}$ such that $-\mathcal{E}=\{x:-x\in\mathcal{E}\}$ contains all of the exponents appearing in the asymptotic expansion of $I(\lambda)$ (described in full detail directly below, in section \ref{E}). Let $\cn(\phi)$ be the Newton polyhedron of $\phi.$ For all positive integer lattice points $\beta$, consider all $c>0$ such that $\beta/c$ lies in $\partial\cn(\phi)$. The set $\mathcal{E}$ (relative to $\cn(\phi)$) is generated by considering $c-n$ over $n$ in a bounded subset of $\N.$ Finally, let $1/t=p_0<p_1<\cdots$ be the well-ordering of $\mathcal{E}.$ We let $d_j$ be the maximum codimension over all faces containing any $\beta\in\N^d$ such that $p_j$ can be written as $p_j=\lf\beta+\bone\rf-n$ for some $n\in\N.$ For example, if $\cn(\phi)$ has a unique compact codimension 1 face that intersects each coordinate axis, we can guarantee that each $d_j=1.$ 
	
	\begin{theorem} Assume $\phi$ is analytic nondegenerate or $k-$nondegenerate. Let $p_0<p_1<\cdots p_j<\cdots$ and $1\le d_j\le d$ be as above for $j\ge 0$ with respect to $\cn(\phi)$. 	
		\begin{itemize}
			\item[(i)] Assume $\phi$ is analytic or smooth. Let $\psi:\R^d\to\R$ be smooth and supported close enough to the origin. Then, there are $a_{j,r}(\psi)\in\C$ such that 
			\eqn{I(\lambda):=\int_{\R^d}e^{i\lambda\phi(x)}\psi(x)dx\sim \sum_{j=0}^n\sum_{r=0}^{d_j-1}a_{j,r}(\psi)\lambda^{-p_j}\log^{d_j-1-r}(\lambda)\label{expansion}}
			for all $\lambda >2$ and all $n\in\N.$ 
			
			\item[(ii)] Assume $\phi\in C^m$ is $k-$nondegenerate and $\psi:\R^d\to\R$ is $C^{m}$ supported close enough to the origin. If $m>(k+1)(d_0+\cdots+d_n)+kp_n+d$,\footnote{This smoothness is not sharp as a byproduct of the proofs used. (Compare to van der Corput in $d=1:$ here we need $C^{k+4}$ instead of $C^k$.)} then there are constants $a_{j,r}(\psi)\in\C$ such that for $\lambda>2$ we have the finite expansion
			\eqnn{I(\lambda)\sim\sum_{j=0}^n\sum_{r=0}^{d_j-1}a_{j,r}(\psi)\lambda^{-p_j}\log^{d_j-1-r}(\lambda).}
			Moreover, the bound for the error term is
			\eqnn{\Bigg{|}I(\lambda)- \sum_{j=0}^n\sum_{r=0}^{d_j-1}a_{j,r}(\psi)\lambda^{-p_j}\log^{d_j-1-r}(\lambda)\Bigg{|}\lesssim \lambda^{-p_{n+1}}\log^{d'-1}(\lambda)}
			where $d'$ is the largest codimension of $\cn(\phi)$ over all faces not contained in coordinate hyperplanes, and the implicit constant is independent of $\lambda\neq 0$.
		\end{itemize}
		\label{thm2}
	\end{theorem}
	As guaranteed, the same estimates hold for all $\lambda\neq 0$ if $\lambda^{-p}$ and $\log^q(\lambda)$ are replaced by $|\lambda|^{-p}$ and $\log^q(|\lambda|+2),$ respectively. Henceforth we assume $\lambda>2$ for notational convenience.
	\section{The Newton polyhedron\label{NewtonChapter}}
	\subsection{Why we need the Newton polyhedron}
	\begin{wrapfigure}{R}{0pt}
		\centering
		\captionsetup{justification=centering}
		\begin{tikzpicture}
		\draw[<->] (0,6.4)--(0,0)--(5,0);
		\draw[fill=lightgray!30,lightgray!30] (1,6.4)--(1,3)--(2,2)--(5,2)--(5,6.4);
		\draw[thick](1,6.4)--(1,3)--(2,2)--(5,2);
		\draw[<->,blue,thick](-.1,6.3)--(2.05,-.15);
		\draw[->,red,thick](1,3)--(2,10/3);
		\node[above right, red] at (2,10/3){$w=(1/2,1/6)$};
		\node[above right, blue] at (2,0){(2,0)};
		\node[below right, blue] at (0,6){(0,6)};	
		\node[below left] at (1,3) {(1,3)};
		\node[below left, blue] at (1.6,1.5) {$H_w$};
		\node[left] at (4,5) {(4,5)};
		\node[above right] at (2,2) {(2,2)};
		\draw[fill] (0,6) circle [radius=.03];
		\draw[fill] (2,0) circle [radius=.03];
		\draw[fill] (1,3) circle [radius=.03];
		\draw[fill] (4,5) circle [radius=.03];
	\end{tikzpicture}
	\caption{A supporting hyperplane of $\cn(\phi)$ for $\phi(x,y)=x^2y^2+xy^3-x^4y^5$.}
	\end{wrapfigure}
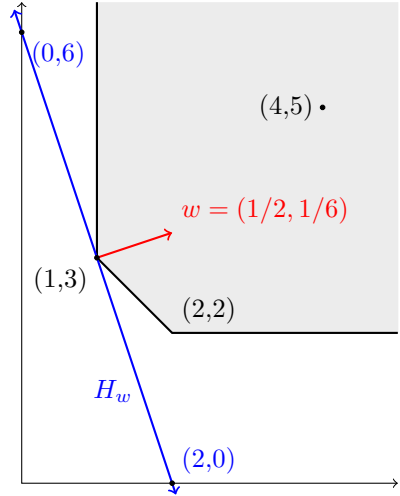
	
	The Newton polyhedron of an analytic function $\phi$ contains all of the information required for determining which monomials are largest near the origin for $x\in(0,1)^d$, which is precisely what Lemma \ref{lem1} quantifies. Given a supporting hyperplane of $\cn(\phi)$ not containing the origin, we can parametrize it by $H_w=\{\xi\in\R^d:\xi\cdot w=1\}$ for some $w\in \R_\ge^d.$ Since $\beta\cdot w\ge 1$ for all $\beta\in\cn(\phi),$ the polyhedron shows $\alpha\in H_w\cap \supp(\phi)$ implies $x^\alpha$ is the largest over all $\alpha\in \supp(\phi)$, under the scalings $x_i^{w_i}=x_j^{w_j}$ over all $1\le i,j\le d$. For example, Figure 2 shows the \textcolor{blue}{hyperplane $H_w$}, supporting at $(1,3)$ parametrized by normal \textcolor{red}{$w=(1/2, 1/6)$}. One can \textit{see} that $x^4y^5$ can never be larger than $\max\{x^2y^2, xy^3\}$ since $(4,5)$ does not lie on a compact face. 
	
	Although the Newton polyhedron contains information about the largest monomials, cancellation between monomials prevents Lemma \ref{lem1} from being true in general. The nondegeneracy condition guarantees that the Newton polyhedron gives us the necessary information about the behavior of $\nabla\phi$. The equivalence from Lemma \ref{lem1} provides an alternative, analytic way to impose nondegeneracy.
	
	
	
	\subsection{Normal vectors of the Newton polyhedron}
	We briefly discuss a subset of linear functionals on $\R^d$, namely the set 
	\eqnn{\{w\in\R_\ge^d: \xi \cdot w =1\text{ for some  }\xi\in \partial \cn(\phi)\}.} 
	For Lemma \ref{lem1}, we care about the whole set, but afterwards we will only care about those finitely many $w$ corresponding to codimension 1 faces not contained in coordinate hyperplanes. From now on such faces will be called \textbf{facets}. Additionally, \textbf{supporting hyperplanes} $H$ of $\cn(\phi)$ refer \textit{only to those supporting hyperplanes not containing the origin}, so that any supporting hyperplane $H$ can be parametrized by $\{\xi : \xi \cdot w=1\}=:H_w.$ There is a nice geometric way to show such normals exist: if $w=(w_1,\dotsc, w_d),$ then $H_w$ intersects the coordinate axes at $x_i=w_i^{-1}$ whenever $w_i\neq 0$, and does not intersect the $x_i$ axis if $w_i=0$, exactly as in Figure 2. This implies $w_j=0$ if and only if $w$ is a normal to a hyperplane intersecting $\cn(\phi)$ in some unbounded face. The following section explains one important reason why we consider such a parametrization.
			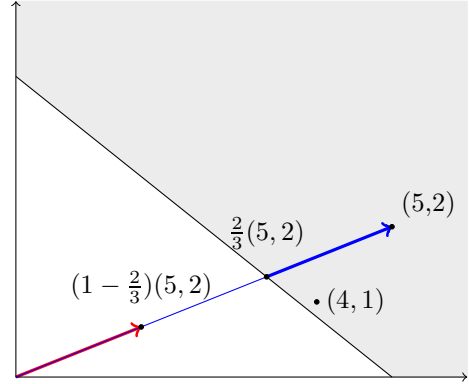
\begin{wrapfigure}{R}{0pt}
				\centering
				\captionsetup{justification=centering}
				\begin{tikzpicture}
				\draw[fill=lightgray!30,lightgray!30] (0,4)--(0,5)--(6,5)--(6,0)--(5,0);
				\draw[->,very thick,red] (0,0)--(5/3,2/3);
				\draw[<->] (0,5)--(0,0)--(6,0);
				\draw (0,4)--(5,0);
				\draw[->,blue] (0,0)--(5,2);
				\draw[->,very thick,blue] (10/3,4/3)--(5,2);
				\node[above right] at (5,2) {(5,2)};
				\draw[fill] (5,2) circle [radius=.03];
				\draw[fill] (4,1) circle [radius=.03];
				\draw[fill] (10/3,4/3) circle [radius=.03];
				\draw[fill] (5/3,2/3) circle [radius=.03];
				\node[above=.25cm] at (10/3,4/3) {$\frac{2}{3}(5, 2)$};
				\node[above=.2cm] at (5/3,2/3) {$(1-\frac{2}{3})(5, 2)$};
				\node[right] at (4,1) {$(4, 1)$};
				\end{tikzpicture}
				\caption{An exponent of the asymptotic expansion of $I(\lambda)$ for $\phi(x,y)=x^5+y^4+x^4y$ not simply of the form $-\lf\alpha+\bone\rf$: the scaling is that of $\frac{1}{3}(5,2)$ instead of $(5,2).$\vspace{1cm}}
			\end{wrapfigure}
			
		\subsection{The set $\mathcal{E}$ of exponents in Theorem \ref{thm2}\label{E}}
		In order to understand geometrically the exponents appearing in the expansion, we define some geometric notions corresponding to scalings in terms of the Newton polyhedron. In this section, let $W$ be the finite set of all normals corresponding to facets of $\cn(\phi).$ For $\alpha\in \N^d$ define
		\eqnn{\hspace{-6.5cm}\lf\alpha\rf=\min_{w\in W} \alpha\cdot w,}
		and define the minimizing set
		\eqnn{\vspace{.04cm}\hspace{-6.5cm}\bn(\alpha)= \{w\in W: \alpha\cdot w = \lf\alpha\rf\}.}
		Note that $\alpha/\lf\alpha\rf\in\partial\cn(\phi)$  if $\lf\alpha\rf\neq 0,$ and $\alpha$ lies in a coordinate hyperplane if $\lf\alpha\rf=0.$ Thus, $\lf\alpha\rf$ can be defined independently of our parametrization. All exponents appearing in the asymptotic expansion are contained in the set $-\mathcal{E}=\{n-\lf\alpha+\bone\rf:\alpha\in\N^d, n\in\N\}$, i.e., if $\lambda^{-p}$ has nonzero coefficient then $p$ must lie in $\mathcal{E}.$ More precisely, if $\lf\alpha+\bone\rf-n\in\mathcal{E}$ then we can write $\alpha=\alpha^1+\cdots+\alpha^n$ for some $\alpha^i\in\cn(\phi)\cap\N^d.$ In particular, Proposition \ref{asymprop:p:2} will guarantee that $\lf\beta+\bone\rf-n\ge 1/t,$ since it shows $\lf\alpha+\beta\rf\ge \lf\alpha\rf+\lf\beta\rf$ for all $\alpha,\beta\in\N^d$. By a classical result answering when an arbitrary integer lattice point inside a cone is a  positive integer sum of the generators of the cone, one can reduce \lq\lq $n\in\N$\rq\rq to \lq\lq $n\in \{0,1,\dots, c_d\}$\rq\rq where $c_d$ depends only on the dimension (the $+\bone$ adds a minor difficulty to this classical question). 
		
		Figure 3 shows the Newton polyhedron of $\phi(x,y)=x^5+y^4+x^3y^2.$ By Theorem \ref{thm2}, the first exponent in the expansion is predicted to be $-1/t=-9/20$. The second is predicted to be $-1/2$: indeed $(5,2)$ provides us such an example: $(5,2)=(4,1)+\bone$ with $(4,1)\in\cn(\phi)$ and $1-(5,2)\cdot(1/5,1/4)=-1/2.$ Next, $(4,3)=(3,2)+\bone$ with $n=1$ produces $-11/20,$ which again cannot be written as $-(\beta+\bone)\cdot (1/5,1/4)$ for any $\beta\in\N^d.$ The method in Cho-Kamimoto-Nose\cite{kn16toric} predicts a noticeably larger second term $\lambda^{-7/15}>\lambda^{-1/2}$ in this case (but additional elementary methods can be used to show the coefficient of this term is zero).
		
		\subsection{Further clarifications regarding Theorem \ref{thm2}}
		Some statements of Theorem \ref{thm2} still require clarification. First, we show that $\{\lf\alpha+\bone\rf\}_{\alpha\in\N^d}$ runs through finitely many arithmetic progressions of positive rationals, which implies $\mathcal{E}$ is well-ordered. Each normal $w$ of a facet $F$ of $\cn(\phi)$ can be uniquely defined by $d$ linearly independent vectors $\alpha^i$ in $\supp(\phi)\cap F$. If $A$ is the matrix with rows $\alpha^i,$ then by definition, $w$ satisfies $Aw=\bone$. Hence, $w=A^{-1}\bone.$ The matrix $A^{-1}$ must have rational entries, since $A$ has rational entries, therefore $w\in\Q^d.$  By geometric considerations, $w$ must have nonnegative entries. The Newton polyhedron has finitely many facets, so there are finitely many such $w$. Writing each component $w_j$ of a fixed normal $w$ as $w_j=r_j/q_j$ where $r_j, q_j$ are integers, let $q_w$ be the lowest common multiple of the $q_j$ corresponding to $w.$ The arithmetic progressions are generated by the rationals $1/q_w$ over all normals $w$ of facets of $\cn(\phi).$ In the example corresponding to Figure 3, indeed $20=\text{lcm}(4,5).$
		
		If $t>1,$ Varchenko showed that the first term of the expansion (\ref{expansion}) with nonzero coefficient is $\lambda^{-1/t}\log^{d_0-1}(\lambda),$ where $d_0$ is the largest codimension over all faces containing $\mathbf{t}$. Indeed, this is the first term guaranteed by Theorem \ref{thm2} since $\lf\bt\rf=1$ implies $\lf \bone\rf=1/t$. Moreover, before the proof of Theorem \ref{thm2}, we will have developed the necessary techniques to show that $n-\lf\beta+\bone\rf=-1/t$ implies $(\beta+\bone)/\lf\beta+\bone\rf$ must lie on a face of codimension at most $d_0.$

	\section{Proof of Lemma \ref{lem1}}\label{POL}
	All results build on Lemma \ref{lem1}, which says that nondegeneracy of $\phi$ implies the sharpest possible decay rate for $\nabla\phi$ near the origin. The lemma closely resembles Lojasiewicz's famous theorem about decay of analytic functions near their singularities\cite{lojasiewicz59}, and is also closely related Lemma 3.6 (for nondegenerate analytic and smooth functions) of Greenblatt\cite{greenblatt12}, but does not follow from either result. The lemma is also very closely related to work of Yoshinaga\cite{yoshi89}. 
	
	For the rest of the section, we consider nondegenerate $\phi:\R^d\to\R$ and we write $\phi=P_k+R_k,$ where \eqnn{P_k(x)=\sum_{|\alpha|\le k} c_{\alpha}x^{\alpha},\text{ and } R_k(x)=\sum_{|\alpha|= k} h_{\alpha}(x)x^{\alpha}}
	are guaranteed by Taylor's theorem. We take $k$ such that $\cn(\phi)=\cn(P_k)$: for analytic functions, $k$ is guaranteed because the Newton polyhedron only has finitely many extreme points, and for $k-$nondegenerate $\phi$ we are considering only convenient $\cn(P_k)$, so that $\cn(\phi)$ still accurately describes the decay of $\phi$. This can also be accomplished assuming other \lq\lq finite type\rq\rq conditions. Exactly what we need in order for the Newton polyhedron to be useful for $C^k$ phases is summarized by the following nice property of analytic functions: we can choose $h_\alpha\equiv 0$ in the remainder whenever $|\alpha|=k$ is such that $\alpha\notin\cn(\phi)$.
	
	By Taylor's theorem,
	\eqn{x_j\partial_j\phi(x)= \sum_{|\alpha|\le k} c_{\alpha}' x^{\alpha} + \sum_{|\alpha|= k} h_{\alpha}'(x) x^{\alpha}\label{phij}}
	for all $1\le j\le d,$ where $c_{\alpha}'=\alpha_jc_{\alpha},$ and $h_{\alpha}'$ depends on $j$. For each compact $F\subset \cn(\phi),$ we write $x_j\partial_j P_k(x)$ in (\ref{phij}) as
	\eqn{\sum_{\alpha\in F}c_{\alpha}'x^{\alpha} + \sum_{\alpha\notin F}c_{\alpha}'x^{\alpha}.\label{splitsum}}
	The main goal is to show for all $\varep$ small enough there is a compact $F$ so that the significant contribution of $\varep_j\partial_j\phi(\varep)$ comes from the polynomial on the left of (\ref{splitsum}) for some $1\le j\le d$.
	
	\subsection{Supporting hyperplanes of $\cn(\phi)$ and scaling}
	The following proposition is used to define some constants necessary for applying nondegeneracy to (\ref{phij}).
	
	\begin{proposition}\label{scaleprop}
		Let $\varep\in (0,1)^d$ and let $\beta, \alpha^1,\dotsc, \alpha^{n}\in\R^d$ be linearly independent. Assume for all $1\le i\le n$ there is a positive $C<1$ such that $C\varep^{\beta}\le \varep^{\alpha^i}\le \varep^{\beta}.$ There is some $b\in (0,1)$ depending only on $\alpha^1,\dotsc, \alpha^n$ and $C,$ such that for some $y\in [b,b^{-1}]^d$ and all $1\le i\le n$ we have
		\eqn{y^{\alpha^i}=\varep^{\alpha^i-\beta}.\label{eqlts}}
		Moreover, if $\alpha=\sum \lambda_i \alpha^i$ for $\sum\lambda_i=1,$ then 
		\eqn{y^\alpha=\varep^{\alpha-\beta}\label{scaleprop2}}
	\end{proposition}
	
	\begin{proof}
		Let $A$ be the $n\times d$ matrix with rows $\alpha^1,\dotsc,\alpha^{n}$. Without loss of generality, assume that the first $n$ columns of $A$ are linearly independent. Let $v\in\R^n$ be the vector defined componentwise by $v_i=\log(\varep^{\alpha^i-\beta})$. Consider $\tilde{A}u=v$ where $\tilde{A}=(\alpha_i^j)_{1\le i,j\le n}.$ Since $\tilde{A}$ has full rank, we can solve $u=\tilde{A}^{-1}v$. Writing $\rho=\|\tilde{A}^{-1}\|_\infty,$ we bound
		\eqnn{\|u\|_\infty\le \|\tilde{A}^{-1}\|_\infty \|v\|_1 \le \rho\|v\|_1 .} 
		Therefore $-\|v\|_1\rho\le u_i\le \|v\|_1\rho$ for all $i.$ Now we can find precisely which  box we seek:
		\eqnn{C^{d\rho}\le \varep^{\rho(\alpha^1+\cdots+\alpha^n-n\beta)} \le 2^{u_i}\le  \varep^{-\rho(\alpha^1+\cdots+\alpha^n-n\beta)}\le C^{-d\rho }.}
		Hence, letting $b=C^{d\rho}\in (0,1)$, we see that the vector $y\in [b,b^{-1}]^d$ defined by $y=2^{u}$ satisfies the system of equations (\ref{eqlts}). Finally, (\ref{scaleprop2}) follows from rewriting $\alpha-\beta=\sum\lambda_i(\alpha^i-\beta).$
	\end{proof}
	The last part of the proposition is useful because $F\cap \supp(P_k)$ might not be a linearly independent set, but it is always contained in the affine hull of $\dim(F)+1$ many linearly independent vectors contained in $F.$
	
	\subsection{The main proposition: circumventing resolution of singularities}
	Motivated by Proposition \ref{scaleprop}, we define constants required to talk about scaling over faces $F\subset\cn(\phi)$ in order to apply nondegeneracy. 
	
	For any facet $F$ of $\cn(\phi)$ and linearly independent $\alpha^1,\cdots, \alpha^n\in \supp(P_k)\cap F,$ define $A$ to be the $n\times d$ matrix with rows $\alpha^i$ and for each $A$ pick a full rank $n\times n$ submatrix $\tilde{A}$, defined by taking $n$ independent columns of $A.$ Define the constant
	\eqnn{\rho=\max_{A} \|\tilde{A}^{-1}\|_\infty \in (0,\infty),}
	where the maximum is taken over all finitely many possible $A$ ($\supp(P_k)$ is finite). Define the positive real number $a$ to be the maximum
	\eqnn{a=2\max_{1\le j\le d}\sum_{|\alpha|\le k} |c_{\alpha}'| 4^{|\alpha|}.}
	We define the positive constant $C_0$ by
	\eqnn{C_0=\min_{\substack{F\subset \cn(\phi)\\\text{compact}}}\inf_{x\in [1,4]^d}\|x\nabla\phi_F(x)\|.}
	Since $\cn(\phi)=\cn(P_k),$ by definition of $a$ and nondegeneracy of $\phi$, $0<C_0\le a/2.$ Now for $1\le m\le d,$ recursively define the constants
	\eqn{b_{m}=\Big(\frac{C_{m-1}}{a}\Big)^{d\rho},\label{bdef}}
	\eqnn{C'_{m}=\min_{\substack{F\subset \cn(\phi)\\ \text{compact}}}\inf_{x\in [b_{m}, 4b_{m}^{-1}]^d}\|x\nabla\phi_F(x)\|,}
	and finally,
	\eqn{C_{m}=\min\{C'_{m}, C_{m-1}/a\}.\label{cdef}}
	Letting $b_0=1,$ it is easy to see $a>C_0> C_1> \cdots > C_{d-1}>C_d>0$ and therefore $b_0>b_1> \cdots > b_{d-1}>b_d>0.$
	
	If $u\in\cn(\phi)$ does not lie in a compact face, then we can write 
	\eqn{u=v_{u}+\gamma_{u}\label{p}}
	for some $v_{u}$ lying in a compact face and $\gamma_{u}\in\R_\ge^d,$ by definition of polyhedron.\footnote{This follows from Theorem 1.2 of Ziegler's textbook\cite{ziegler95}. To see that the Newton polyhedron is a polyhedron in the sense of Ziegler, one can read Proposition 1 in \cite{thesis}.} Since $\supp(P_k)$ is finite, we can define 
	\eqn{p=\min_{u\in \supp(P_k)} \{\|\gamma_{u}\|,1\}\label{pp}}
	where the minimum is over only those $u$ not lying on any bounded face. One last constant used in the proof of the main proposition requires definition. Recall the parametrization of supporting hyperplanes $H$ not containing the origin: $H=H_{w}=\{\xi  : \xi\cdot w=1\}.$ Define
	\eqnn{\delta'=-1+\min_{\alpha^1,\alpha^2\in \supp(P_k)}\inf_{w}  \Bigg(\frac{\alpha^1+\alpha^2}{2}\cdot w\Bigg)}
	where the minimum is taken over all $\alpha^1, \alpha^2$ not contained in the same facet, and the infimum is taken over all normals $w$ of $\cn(\phi)$. In the case where there exist such $\alpha^1, \alpha^2,$ we claim that $\delta'>0$. Since all $\alpha\in \cn(\phi)$ and all normals $w$ to $\cn(\phi)$ satisfy $\alpha\cdot w \ge 1,$ this simply follows by convexity of $\cn(\phi):$ $\alpha^1$ and $\alpha^2$ must lie on some nontrivial line segment contained in $\cn(\phi)$ but not in $\partial\cn(\phi)$. If there are no such $\alpha^1, \alpha^2,$ let $\delta'=1.$ Define 
	\eqn{\delta=\min\{p,\delta'p\}.\label{dd}}
	
	After proving one more minor result, we can move on to the main proposition required for estimating $x\nabla\phi(x)$.
	
	\begin{proposition}
		Let $j\in\N^d$ satisfy $j_i \ge c>0$ for all $1\le i\le d.$ Assume $w$ is the normal to a supporting hyperplane $H_{w}$ of $\cn(\phi)$ such that $w$ is parallel to $j.$ Then, 
		\eqnn{\frac{|j|}{|w|}\ge c.}\label{propq}
	\end{proposition}
	\begin{proof}
		Since each component of $j$ is nonzero, indeed there is a normal $w$ to a supporting hyperplane $H_{w}$ of $\cn(\phi)$ that is a constant multiple of $j.$ Let $\alpha \in \supp(\phi)\cap H_{w}.$ Then, since $w/|w|=j/|j|,$ and some component $\alpha_i$ of $\alpha$ is nonzero,
		\eqnn{ \frac{|j|}{|w|} = \alpha \cdot w \frac{|j|}{|w|} =\alpha\cdot j\ge \alpha_i c\ge c.}
	\end{proof}

	\begin{proposition}[Main proposition] Let $P\not\equiv 0$ be a polynomial satisfying $P(\bz)=0.$ In terms of $P$, define the constants $b_m$, $C_m,$ $p$ and $\delta$ as in (\ref{bdef}), (\ref{cdef}), (\ref{pp}), and (\ref{dd}) respectively, for $0\le m\le d$. Fix $j\in\N^d$ with nonnegative components such that $\|2^{-j}\|<(C_{d}/a)^{1/\delta}$. There is a compact face $F_0$ such that $\varep^{\beta}$ is maximized when $\beta\in F_0$ over all $\alpha\in \cn(\phi)$ for $\varep=2^{-j}.$ Moreover, there is a compact face $F'\supseteq F_0$, and $0\le m'\le d$ such that
		
		\begin{itemize}
			\item[(i)] For all $v\in F'$ we have the scaling
			\eqnn{\varep^{v-\beta}=y^{v},\text{ where } y\in [b_{m'}, b_{m'}^{-1}]^d, \text{ and }}
			\item[(ii)] for all $u\in\supp(P)-F'$ we have the upper bound
			\eqnn{\varep^{u}\le \varep^{\beta} C_{m'}/a.} 
		\end{itemize}\label{mainprop}
	\end{proposition}
	
	\begin{proof}
		First, we claim that $\varep^{\beta}$ is maximized if $\beta\in H_{w}$, where $w$ is a multiple of $j.$ For all $\alpha\in\cn(\phi)-H_{w},$ we know $ \alpha\cdot w >1.$ Therefore,
		\eqnn{\varep^{\alpha} = 2^{- \alpha\cdot j} <2^{-\beta\cdot j} =\varep^{\beta}.}
		Let $F_0=H_{w}\cap \cn(\phi)$ and fix $\beta\in F_0.$ 
		
		If every $u\notin F_0$ satisfies $\varep^{u}\le \varep^{\beta} C_0/a,$ $(ii)$ is trivial and $(i)$ is also clear: let $F'=F_0$ and $y=\bone.$ Otherwise, for $0\le m\le d-1$ define 			
		\eqnn{\Lambda_m=\{u\in\supp(P): \varep^{u}>\varep^{\beta} C_m/a\}.}
		Each $\Lambda_m$ is nonempty because each $\Lambda_m$ contains $\beta$. Let us first show that $\Lambda_{m}$ is contained in a single facet. If some $u^1,u^2\in \Lambda_{m}$ do not lie in the same facet,
		\eqnn{(u^1+ u^2)\cdot w -2 \ge 2\delta'\ge 2\delta/p\ge 2\delta.} 
		Therefore $\delta\le u^i\cdot w-1 =(u^i-\beta)\cdot w$ for some $i.$ So by Proposition \ref{propq} and our assumptions on $j,$
		\eqnn{ \varep^{u^i-\beta}=2^{- (u^i-\beta)\cdot j}\le 2^{-\delta |j|/|w|} <C_d/a<C_m/a.}
		So indeed $\Lambda_{m}$ is contained in a single facet of $\cn(P).$
		
		Next we show that $\Lambda_{m}$ is contained in some compact face $F_{m}$ of $\cn(\phi)$. Otherwise, there is some $u\in \Lambda_{m}$ not lying in any compact face. Write $u=v_{u}+\gamma_{u}$ as in (\ref{p}). Since $v_{u}\in\cn(P),$ by definition of $p\ge \delta$,
		\eqnn{\varep^{u-\beta} = \varep^{v_{u}-\beta+\gamma_{u}}\le \varep^{\gamma_{u}}\le (C_{d}/a)^{dp/\delta}\le  C_{d}/a< C_{m}/a.}
		
		Assume $F_{m}\supset F_0$ containing $\Lambda_{m}$ has maximal codimension, i.e., there is an affine basis $\{v^1,\dotsc, v^{\dim(F_{m})+1}\}\subset F_{m}\cap \Lambda_{m}$ for the affine hull of $\Lambda_{m}.$ By Proposition \ref{scaleprop} with $C=C_m/a$, there is some $y\in [b_{m+1}, b_{m+1}^{-1}]^d$ so that for all $1\le i\le \dim(F_{m})+1$ the equalities $\varep^{v^i-\beta}=y^{v^i}$ hold. The proposition also tells us that $\varep^{v-\beta}=y^{v}$ for all $v\in F_m.$ Since this holds over all $0\le m\le d-1,$ we are left with claim $(ii).$ Assuming there is some $0\le m<d-1$ such that $\Lambda_m=\Lambda_{m+1},$ this claim is obvious by applying Proposition \ref{scaleprop} and letting $m'=m+1$. If there is no such $m,$ notice that the ordered set $\{\dim(F_0), \dim(F_1),...\}$ is strictly increasing and bounded strictly above by $d,$ so in particular $\dim(F_m)\ge m$ for all $0\le m\le d-1.$ In this case we see that $\dim(F_{d-1})=d-1.$ Therefore $m'=d$ satisfies property $(i)$; property $(ii)$ is obvious by definition of $\delta'$ and the bounds assumed on $\varep=2^{-j}.$
	\end{proof}
	
	Now Proposition \ref{mainprop} is applied to finish proving the main lemma. Write $\phi(\varep)$ as
	\aln{& \sum_{\alpha\in F_{m'}} c_{\alpha}' \varep^{\alpha}+\sum_{\alpha\notin F_{m'}} c_{\alpha}' \varep^{\alpha} + \sum_{|\alpha|=k}h_{\alpha}'(\varep)\varep^{\alpha}
		\nonumber\\  
		& =\sum_{\alpha\in F_{m'}}c_{\alpha}' \varep^{\alpha}+\sum_{\alpha\notin F_{m'}}c_{\alpha}' \varep^{\alpha}+\sum_{|\alpha|=k}h_{\alpha}'(\varep)\varep^{\alpha}\nonumber\\
		& \stackrel{\text{Prop. \ref{scaleprop}}}{=}\Bigg{(}\sum_{\alpha\in F_{m'}}c_{\alpha}' y^{\alpha} +\sum_{\alpha\notin F_{m'}}c_{\alpha}' \varep^{\alpha-\beta}+\sum_{|\alpha|=k}h_{\alpha}'(\varep)\varep^{\alpha-\beta} \Bigg{)}\varep^\beta.\label{sums}}
	For the leftmost sum of (\ref{sums}), we choose $1\le j\le d$ such that by nondegeneracy,
	\eqn{\Bigg{|}\sum_{\alpha\in F_{m'}}c_{\alpha}' y^{\alpha}\Bigg{|}\ge C_{m'}',\label{Fbound}}
	which is possible because $y\in [b_{m'}, b_{m'}^{-1}]$. For the second term, Proposition \ref{mainprop} guarantees $\varep^{\alpha-\beta}\le C_{m'}/a$ for monomials $\alpha\notin F_{m'}$ appearing in $P_k$. Also, by the definition of $a,$ we know
	\eqn{\sum_{\alpha\notin F_{m'}}|c_{\alpha}'|\le a/2.\label{nFbound}}
	To estimate the remainder term, let $\zeta>0$ be small enough so that for all $\|\varep\| <\zeta$ and all $1\le j\le d$ the remainders $h_\alpha'=h_{\alpha,j}$ can be bounded by
	\eqn{ \sum_{|\alpha|=k} |h_{\alpha}'(\varep)| \le \frac{C_{d}}{4}.\label{Rbd}}
	This is possible since each $h_{\alpha}'$ goes to $0$ as $\varep\to \bz.$ Now bound 
	\eqn{\Bigg{|}\sum_{\alpha\in F_{m'}}c_{\alpha}' y^{\alpha}+\sum_{\alpha\notin F_{m'}}c_{\alpha}' \varep^{\alpha-\beta}+\sum_{|\alpha|=k}h_{\alpha}'(\varep)\mathbf \varep^{\alpha-\beta}\Bigg{|}.\label{tequal}}
	simply by applying the triangle inequality. Here we must use that all monomials $\varep^{\alpha-\beta}$ appearing in the remainder are such that $\alpha\in \cn(\phi)$. This assumption is necessary to guarantee $\alpha$ does not lie below the supporting hyperplane $H_w\ni\beta$ of $\cn(\phi)$: $\alpha\cdot w\ge 1$ implies $\varep^{\alpha}\le \varep^\beta$. Therefore, by (\ref{Fbound}), (\ref{nFbound}), and (\ref{Rbd}), the quantity (\ref{tequal}) is bounded below by
	\alnn{& \Bigg{|}\sum_{\alpha\in F_{m'}}c_{\alpha}' y^{\alpha}\Bigg{|}- \frac{C_{m'}}{a}\cdot \sum_{\alpha\notin F_{m'}} |c_{\alpha}'| - \sum_{|\alpha|=k} |h_{\alpha}'(\varep)|\\
		& \ge C_{m'}' - \frac{C_{m'}}{a}\cdot \frac{a}{2}-  \frac{C_{d}}{4} =C_{m'}'-\frac{C_{m'}}{2}- \frac{C_{d}}{4}\\
		& \ge\frac{C_d}{4},}
	where the last inequality follows from $C_{d}< C_{m'}\le C_{m'}'.$ Therefore for dyadic $\varep=2^{-j},$ we have $\|\varep\nabla\phi(\varep)\|\ge \varep^\beta C_d/4.$
	
	We summarize for all future (and past) use that $x$ small enough means for all $1\le i\le d,$
	\eqn{0<x_i<s := \min\{(C_{d}/a)^{1/\delta}, \zeta\}.\label{smallenough}}
	To finish the proof of Lemma \ref{lem1}, the result above needs to be extended to all real $x$ small enough. The first step is to observe that nowhere in Propositions \ref{propq} and \ref{mainprop} did we use that $j\in\N^d;$ we just required that each component of $j$ is large enough $(\varep=2^{-j}$ small enough). Finally, for all $x\in [2^{-j}, 2^{-j+\btwo}],$ there is a uniform constant depending on the extreme points of $\cn(\phi)$ but independent of $j$ such that $x^\beta\gtrsim \varep^\beta$ for all $\beta.$
	
	\section{Proof of Lemma \ref{thm1} \label{varchestimate}}
	\subsection{Estimating an integration by parts operator}
	Let $\phi$ be $k-$nondegenerate or analytic nondegenerate on $[-1,1]^d$ and let $\eta\in C^{k}$ be supported in $[1,4]^d.$ The goal of this section is to integrate 
	\eqnn{I_+(\lambda, \varep)=\int_{[1, 4]^d}e^{i\lambda\phi(\varep x)}\eta(x)dx}
	by parts $k$ times in order to get good estimates on $I_+(\lambda,\varep)$. 
	
	Define $f(x)=\frac{\nabla\phi(x)}{\|\nabla\phi(x)\|^2}$. Since $\nabla\phi(x)\neq\bz$ away from coordinate axes, $f(x)$ is $C^{k-1}$ in each component away from coordinate axes. Define the operator $D=D_{\varep,\phi}$ by 
	\eqnn{D(g)(x)=\frac{\nabla g(x)\cdot f(\varep x)}{i\lambda}.}
	We can check that $D$ fixes $e^{i\lambda\phi(\varep x)}$. If $g$ is $C^k,$ we can estimate $(D^t)^N(g)(x)$ for $1\le N\le k-1$, where the adjoint $D^t$ of $D$ is given by the divergence
	\eqnn{D^t(g)(x)=-\nabla\cdot\frac{g(x)f(\varep x)}{i\lambda}.}
	To estimate $(D^t)^N(g)$ we consider the components $f_n$ of $f.$ The goal is to show $\partial^{\beta}f_n$ is a linear combination of terms of the form
	\eqn{\frac{\partial^{\gamma^1}\phi\cdots \partial^{\gamma^{2^r-1}}\phi}{\|\nabla\phi\|^{2^r}}=\frac{\partial^{\gamma^1}\phi}{\|\nabla\phi\|}\cdots\frac{\partial^{\gamma^{2^r-1}}\phi}{\|\nabla\phi\|}\|\nabla\phi\|^{-1},\label{formulaf}}
	where $0\le |\gamma^\ell|\le N$ and $r=|\beta|$. Assuming this is true, we first show that
	\eqn{|\partial^\beta f_n(\varep x)|\lesssim \varep^{-\alpha}\label{varepbound}}
	for any $\beta\in\N^d$ and any $\alpha\in \cn(\phi),$ for $\varep$ small enough: (\ref{formulaf}) guarantees $\partial^\beta f_n(\varep x)$ is a linear combination of products of $2^r-1$ terms 
	\eqn{\varep^{\gamma^\ell}\partial^{\gamma^\ell}\phi(\varep x)\|\varep\nabla\phi(\varep x)\|^{-1}\label{starrr}} 
	times a single $\|\varep\nabla\phi(\varep x)\|^{-1}$. We claim the first $2^r-1$ terms (\ref{starrr}) can be bounded above by a constant independent of $\varep$, while  $\|\varep\nabla\phi(\varep x)\|^{-1}$ is bounded above by $\varep^{-\alpha}$ for any $\alpha\in \cn(\phi)$ by Lemma \ref{lem1}. Indeed, if $x\in[1,4]^d,$ then the absolute value of $\varep^{\gamma^\ell}\partial^{\gamma^\ell}\phi(\varep x)=x^{-\gamma^\ell}(\varep x)^{\gamma^\ell}\partial^{\gamma^\ell}\phi(\varep x)$ is bounded above by a uniform constant depending on $\phi$ times $\varep^v$ for some $\alpha\in \cn(\phi)$ depending on $x$ and $\varep$ (consider the Taylor expansion). Lemma \ref{lem1} guarantees each term (\ref{starrr}) is indeed bounded above by a constant independent of $\varep$ since $\|\varep\nabla\phi(\varep x)\|^{-1}\le \|\varep x\nabla\phi(\varep x)\|^{-1}\lesssim \varep^{-\alpha}$.
	
	We now begin the proof that $\partial^\beta f_n$ is a linear combination of terms of the form (\ref{formulaf}). The proof is by induction on $|\beta|$. The base case is clear by the discussion above. Taking the partial derivative of $\|\nabla\phi\|^{2^r},$ with respect to $x_j,$
	\eqnn{\partial_j\|\nabla\phi\|^{2^r}=2^{r}\|\nabla\phi\|^{2^{r}-2}\sum_{\ell=1}^d\phi_{x_\ell}\phi_{x_\ell x_j},}
	which is a sum of products of $(2^r-2)+2=2^r$ functions, each equal to some partial derivative of $\phi$ of order no more than 2. Writing $\beta=\gamma^1+\cdots +\gamma^{2^r-1},$ the function
	\eqnn{\partial_j\sum_{\beta}a_\beta \partial^{\gamma^1}\phi\cdots \partial^{\gamma^{2^r-1}}\phi =\sum_{\beta}\sum_{l=1}^{2^r-1}a_\beta\partial^{\gamma^1}\phi\cdots\partial^{\gamma^l+\bee_j}\phi\cdots\partial^{\gamma^{2^r-1}}\phi}
	is again a sum of products of $2^r-1$ functions, each equal to some partial derivative of $\phi$ of order at most one more than $|\beta|.$ Therefore the numerator of
	\eqnn{\partial_j\frac{\sum_{\beta} \alpha_\beta\partial^{\gamma^1}\phi\cdots \partial^{\gamma^{2^{r}-1}}\phi}{\|\nabla \phi\|^{2^r}}}
	is equal to 
	\eqnn{
	\begin{split}
	& \|\nabla\phi\|^{2^r}\sum_{\beta}  \sum_{l=1}^{2^r-1} a_\beta\partial^{\gamma^1}\phi\cdots \partial^{\gamma^l+\bee_j}\phi\cdots \partial^{\gamma^{2^r-1}}\phi\\
	& \hspace{3cm}-\sum_{\beta} a_\beta\partial^{\gamma^1}\phi\cdots \partial^{\gamma^{2^r-1}}\phi\cdot2^{r}\|\nabla\phi\|^{2^{r}-2}\sum_{\ell=1}^d\phi_{x_\ell}\phi_{x_\ell x_j}.
	\end{split}} 
	After reorganizing, we see that we get a sum of products of $2^r+2^r-1=2^{r+1}-1$ functions, each equal to some partial derivative of $\phi$. What's left is the denominator of the partial derivative in the $j$ direction: $\|\nabla\phi\|^{2^{r+1}}.$ So by induction, the proof of (\ref{varepbound}) is complete: let $|\beta|=r>0$ above and write $\beta=\beta'+\bee_j$ for any $j$ such that $\beta_j\neq 0$.
	
	Next, induction can be used to compute that for $\beta^0,\dotsc, \beta^N\in\R^d$ there are constants $a_\beta=a_{\beta^0,\dotsc, \beta^N}\in\{0,1\}$ such that 
	\eqnn{(D^t)^N(g)(x)= (i\lambda)^{-N}\sum_{\substack{1\le j_1,\dotsc, j_N\le d\\|\beta^0+\beta^1+\cdots +\beta^N|=N}} a_{\beta}\partial^{\beta^0}g(x)(\partial^{\beta^1}f_{j_1})(\varep x)\cdots (\partial^{\beta^N}f_{j_N})(\varep x).}
	By (\ref{varepbound}),
	\alnn{& |(D^t)^N(g)(x)|\le \lambda^{-N} \sum_{\substack{1\le j_1,\dotsc, j_N\le d\\|\beta^0+\beta^1+\cdots +\beta^N|=N}} a_{\beta}|\partial^{\beta^0}g(x)|\cdot |(\partial^{\beta^1}f_{i_1})(\varep x)|\cdots |(\partial^{\beta^N}f_{i_N})(\varep x)|\nonumber\\
		& \lesssim \lambda^{-N}\sum_{j_1,\dotsc, j_N=1}^d|\partial^{\beta^0}g(x)|\varep^{-\alpha^{1}}\cdots\varep^{-\alpha^{N}}\label{asdf}}
	for any $\alpha^1,\cdots, \alpha^N\in \cn(\phi).$ In particular, for all $\alpha\in \cn(\phi),$
	\eqn{|(D^t)^N(g)(x)|\lesssim \lambda^{-N}\max_{1\le |\beta^0|\le N}|\partial^{\beta^0} g(x)|\varep^{-N\alpha}\label{dtest}}
	for all $0\le N\le k-1,$ where the implicit constant is independent of $\varep$ and $\lambda.$ 
	
	\subsection{Final estimate for Lemma \ref{thm1}}
	We now put everything together for $\varep$ small enough: writing $(x_1/\varep_1,\dots, x_d/\varep_d)=x/\varep,$
	\alnn{I_+(\lambda,\varep)& =\int_{[\varep,  4\varep]} e^{i\lambda \phi(x)}x^\beta\eta(x/\varep)dx = \varep^\bone \int_{[1,4]^d} e^{i\lambda \phi(\varep x)} (\varep x)^\beta \eta(x)dx\\
	& = \varep^{\beta+\bone}\int_{[1,4]^d} D^N(e^{i\lambda \phi(\varep x)})x^\beta\eta(x)dx = \varep^{\beta+\bone} \int_{[1,4]^d} e^{i\lambda \phi(\varep x)} (D^t)^N(x^\beta\eta(x))dx.}
	By (\ref{dtest}), letting $g(x)=x^\beta\eta(x)\in C^k$ in (\ref{dtest}),
	\eqnn{\int_{[1,4]^d} | (D^t)^N(x^\beta\eta(x))|dx \lesssim \int_{[1,4]^d}\lambda^{-N} \varep^{-N\alpha}dx \lesssim \lambda^{-N} \varep^{-N\alpha}.} 
	Therefore we have proved Lemma \ref{thm1}: for all $0\le N\le k-1$,
	\eqnn{|I_+(\lambda,\varep)|\lesssim\lambda^{-N} \varep^{-N\alpha+\beta+\bone}.}

	\section{Proof of Theorem \ref{wutcor}}	
	We now use Lemma \ref{thm1} and optimization to prove Varchenko's upper bounds. Since we are summing boxes away from coordinate hyperplanes, without loss of generality we can bound over each orthant of $\R^d,$ so only the orthant $\R_\ge^d$ is considered. To obtain a bound on the integral
	\eqnn{I_{+}(\lambda)=\int_{\R_\ge^d} e^{i\lambda \phi(x)} \psi(x)dx=\int_{[0,1]^d} e^{i\lambda \phi(x)} \psi(x)dx, }
	 where $\psi$ is supported in a sufficiently small neighborhood of the origin, a sum is taken over all positive dyadic boxes. To accomplish this, decompose $\psi(x)=\sum_{j_1,\dotsc, j_d=0}^\infty \psi(x)f_{j}(x),$ where $f_j(x)=f(2^{j}x)$ is a partition of unity subordinate to the cover $\{(2^{-j}, 2^{-j+\mathbf{2}})\}_{j\in\N^d}$ of $(0,4)^d$. For existence of such a dyadic partition of unity, see for example Grafakos's textbook\cite{grafakos14}.\footnote{Instead of breaking the integral into orthants, we can consider a multi-dyadic partition of unity similar to Gressman's bi-dyadic partition in \cite{pgs16}. Since Varchenko's bounds are sharp for $t>1,$ and we recover this sharp bound in just orthants, it is an interesting question whether the bounds proven here are sharpest possible over orthants for all $t$.}
	
\subsection{Optimization}\label{LP}
Let $\beta\in \N^d.$ There exist a unique integer $1\le r\le d$ and a unique rational number $\eta >0$ such that $r-1$ is the smallest dimension over all faces of $\cn(\phi)$ containing the vector $v:=\eta(\beta+\bone)$. We show
\eqnn{\sum_{j_1,\dotsc, j_d=0}^\infty \min_{\substack{N\in \N, \\\alpha\in \cn(\phi)}}\{\lambda^{-N}2^{(N\alpha-\beta-\bone)\cdot j}\}\lesssim \lambda^{-1/\eta}\log^{d-r}(\lambda).}
First, considering $N=0,$ we see that for any $1\le i\le d$,
\eqnn{\sum_{j_1=0}^\infty\cdots\sum_{j_i=\log(\lambda)/v_i}^\infty\cdots \sum_{j_d=0}^\infty 2^{-(\beta+\bone)\cdot j}\lesssim \lambda^{-(\beta_i+1)/v_i}=\lambda^{-1/\eta}.}
So it is enough to bound the sum
\eqn{\sum_{j_1=0}^{\log(\lambda)/v_1}\cdots\sum_{j_d=0}^{\log(\lambda)/v_d}  \min_{\substack{N\in \N, \\\alpha\in \cn(\phi)}}\{\lambda^{-N}2^{ (N\alpha-\beta-\bone)\cdot j}\}\label{varsum}}
above by a uniform constant times $\lambda^{-1/\eta}\log^{d-r}(\lambda).$ It is more natural to work in a continuous setting for this problem, so we bound (\ref{varsum}) by a uniform constant times
\eqn{\int_{0}^{\log(\lambda)/v_1}\dotsi \int_{0}^{\log(\lambda)/v_d}  \min_{\substack{N\in \N, \\\alpha\in \cn(\phi)}}\{\lambda^{-N}e^{  (N\alpha -\beta-\bone)\cdot x}\} dx, \label{varint}}
and estimate the integral (\ref{varint}). Since $v$ does not lie in a face of dimensions less than $r-1$, and all components of $v$ are positive, there are linearly independent $\alpha^1,\dotsc, \alpha^r\in F$ whose convex hull contains $v,$ so write 
\eqn{v= \sum_{i=1}^r \lambda_i\alpha^i.\label{convcomb}}
For $1\le i\le r$ let $\theta_i=\lambda_i/(N\eta)$ and $\theta_0= 1-1/(N\eta)$. All $\theta_i$ are positive and their sum is $1$ by the restriction placed on $\eta$, assuming $N>1/\eta$. Moreover,
\eqn{\theta_0(-\beta-\bone) +\sum_{i=1}^r \theta_i (N\alpha^i-\beta-\bone)=\bz.\label{convo}}
The integral (\ref{varint}) can be bounded above by 
\eqn{\int_{0}^{\log(\lambda)/v_1}\dotsi \int_{0}^{\log(\lambda)/v_d} \min_{1\le i\le r} \{e^{-(\beta+\bone)\cdot x}, \lambda^{-N}e^{ (N\alpha^i -\beta-\bone)\cdot x}\} dx.\label{varint2}}
Now change variables: let $x=A^Ty$ where $A$ is the $d\times d$ matrix defined by 
\begin{equation*}
\begin{split}
	A\alpha^i & = \bee_i \text { for } 1\le i \le r,\\
	A \bee_i & = \bee_i \text{ for } r< i\le d.
\end{split}
\end{equation*}
Up to a constant depending on $A	,$ the integral (\ref{varint2}) equals
\eqn{\int_{\substack{ 0\le A^T y\cdot\bee_i\le \log(\lambda)/v_i\\ 1\le i\le d}} \min_{1\le i\le r}\{e^{ -A(\beta+\bone)\cdot y}, \lambda^{-N} e^{A(N\alpha^i-\beta-\bone)\cdot y}\} dy.\label{varint3}}
By (\ref{convcomb}), since $v=\eta(\beta+\bone)$, the vector $A(\beta+\bone)$ can be written as a linear combination of vectors $A\alpha^i=\bee_i$ over $1\le i\le r$. Therefore we can integrate over directions $r<i\le d$ to get
\eqnn{\int_{\substack{ 0\le  A^T y\cdot\bee_i\le \log(\lambda)/v_i\\ r<i\le d}} dy_{r+1}\dotsm dy_d \lesssim \log^{d-r}(\lambda),}
and bound (\ref{varint3}) above by 
\eqn{\log^{d-r}(\lambda) \int_{\R^k} \Big{|}\min_{1\le i\le r}\{e^{ -A(\beta+\bone)\cdot y}, \lambda^{-N} e^{ A(N\alpha^i-\beta-\bone)\cdot y}\}\Big{|} dy_1\dotsm dy_r.\label{varint4}}
Since $A\alpha^i=\bee_i$ and $\sum_{i=1}^r \lambda_i =1,$ there holds
\eqnn{ A(\beta+\bone) \cdot\log(\lambda)\bone = \frac{1}{\eta} \log(\lambda)\sum_{i=1}^r \lambda_i  A\alpha^i\cdot \bone = \frac{1}{\eta} \log(\lambda),}
so $e^{ -A(\beta+\bone) \cdot\log(\lambda)\bone}=\lambda^{-1/\eta}.$ Therefore change variables $y\mapsto y+\log(\lambda)\bone$. This produces a factor $\lambda^{-(\beta+\bone)\cdot w}$, so (\ref{varint4}) is bounded above by $\lambda^{-1/\eta} \log^{d-r}(\lambda)$ times
\eqnn{\int_{\R^k} \Big{|}\min_{1\le i\le r}\{e^{-A(\beta+\bone)\cdot y}, e^{ A(N\alpha^i-\beta-\bone) \cdot y}\}\Big{|} dy_1\dotsm dy_r.}
By (\ref{convo}), 
\eqnn{\mathbf{0} = -\theta_0A(\beta+\bone)+\sum_{i=1}^r \theta_i A(N\alpha^i-\beta-\bone).}
Since $A(N\alpha^i-\beta-\bone)$ for $1\le i\le r$ are linearly independent, 
\eqnn{\sup_{\|y\|_2=1}\min_{1\le i\le r} \{ -A(\beta+\bone)\cdot y, A(N\alpha^i-\beta-\bone)\cdot y\}<0.}
By homogeneity, there is a constant $c=c(\alpha^1,\dotsc, \alpha^k, \beta)>0$ such that
\eqnn{\min_{1\le i\le r} \{-A(\beta+\bone) \cdot y, A(N\alpha^i-\beta-\bone) \cdot y\}<c \|y\|_2.}
After a polar change of variables, we can bound (\ref{varint4}) by a constant times
\eqnn{\lambda^{-1/\eta}\log^{d-r}(\lambda) \int_0^\infty e^{-cr}dr \lesssim \lambda^{-1/\eta}\log^{d-r}(\lambda),}
which is exactly the bound we were looking for.

\subsection{Varchenko's upper bounds as a special case}
Taking $\beta=\bz$ above, observe $\eta= t$. Therefore Theorem \ref{wutcor} generalizes Varchenko's upper bounds in \cite{varchenko76}. However, $C^k$ phases require $k-1>1/t$ in Lemma \ref{thm1} since we used $N>1/t$ above. Since $t\ge 1/d$ is always true for polynomials that vanish at the origin, Varchenko's estimate also holds for all $(d+2)-$nondegenerate phases, and in particular for a class of phases that are merely $C^{d+2}$.
	
\section{Preliminaries for the proof of Theorem \ref{thm2}}
The goal of this section is to prove the asymptotic expansion of the oscillatory integrals under consideration. 

The following properties are required to prove the asymptotic expansion of $I(\lambda)$. All minima $\lf\cdot\rf$ are with respect to $\cn(\phi).$ Properties $(i),(ii)$ can be proven by simply using the definitions of $\lf\cdot\rf$ and $\mathbf{n}(\cdot)$. For $(iii),$ we require that any collection of more than $d-1$ normal vectors $\{w\}$ is such that $\cap_w H_w$ contains at most one point. The last item is a fact about codimension corresponding to the number of facets containing a vector.

\begin{proposition}\label{asymprop:p}
	Let $\alpha, \beta\in \N^d$ be nonzero. Then the following hold:
	\begin{enumerate}[label=(\roman*), ref=\thetheorem(\roman*)]
		\item $\lf c\alpha\rf = c\lf \alpha\rf$ for any $c>0.$\label{asymprop:p:1}
		\item $\lf \alpha +\beta\rf \ge \lf\alpha\rf +\lf\beta\rf$.	\label{asymprop:p:2}
		\item If $\lf\alpha+\beta\rf=\lf\alpha\rf+\lf\beta\rf$, then $\mathbf{n}(\alpha+\beta)=\mathbf{n}(\alpha)\cap \mathbf{n}(\beta).$ Moreover, in this case $\mathbf{n}(\alpha)\neq \mathbf{n}(\beta)$ implies  $|\mathbf{n}(\alpha+\beta)|<\min\{d, |\mathbf{n}(\alpha)|\}.$\label{asymprop:p:3}
		\item $\min\{d,|\bn(\beta+\bone)|\}$ is the greatest codimension over all faces containing $\beta+\bone.$\label{asymprop:p:4}
	\end{enumerate}
\end{proposition}

\subsection{A differential inequality}
We need to show
\eqn{I(\lambda)\sim \sum_{j=0}^n\sum_{r=0}^{d_j-1}a_{j,r}(\psi)\lambda^{-p_j}\log^{d_j-1-r}(\lambda),\label{expansion2}}
 for all $n\in\N$ large enough, as stated in Theorem \ref{thm2} (and will be restated below), and $\lambda>2$ where $p_0<p_1<\cdots$ is the ordering of $\mathcal{E}$. Note that $p_j$ and $d_j$ depend only on the Newton polyhedron of $\phi.$ 
 
 The main goal is to show $I(\lambda)$ satisfies the differential inequality of Lemma \ref{diflem} below. The outline of the rest of the paper is as follows. We first show how this lemma implies Theorem \ref{thm2}. Then it is shown that $I(\lambda)$ satisfies a differential inequality satisfying the hypothesis of the following lemma. 

\begin{lemma}
	Assume $f: (2,\infty)\to\C$ is smooth. Assume there are real numbers $0<p_0<p_1<\cdots< p_{n+1}$ and positive integers $d_0,\cdots, d_{n+1}$ such that
	\eqnn{\Bigg{|}\Bigg{(}\lambda\frac{d}{d\lambda}+p_n\Bigg{)}^{d_n}\cdots \Bigg{(}\lambda\frac{d}{d\lambda}+p_0\Bigg{)}^{d_0} f(\lambda)\Bigg{|}\lesssim \lambda^{-p_{n+1}}\log^{d_{n+1}-1}(\lambda).\label{difine}}
	Then, there are constants $a_{j,k}\in\C$ such that
	\eqnn{f(\lambda)=\sum_{j=0}^n\sum_{k=0}^{d_n-1}a_{j,k}\lambda^{-p_j}\log^{d_j-1-k}(\lambda) + O\Big(\lambda^{-p_{n+1}}\log^{d_{n+1}-1}(\lambda)\Big).}
	\label{diflem}
\end{lemma}

We prove this lemma with a series of elementary propositions. Their details are left out, but can be found in \cite{thesis}. For example, Proposition \ref{simpprop} can be shown by an induction argument on $0\le m\le d_0+\cdots+ d_n.$

\begin{proposition}
	Let $h: (2,\infty)\to \C$ be smooth. Assume there are positive reals $p_0<p_1<\cdots< p_{n}$ and positive integers $d_0,\cdots, d_{n}$ such that
	\eqnn{\Bigg{(}\lambda\frac{d}{d\lambda}+p_n\Bigg{)}^{d_n}\cdots \Bigg{(}\lambda\frac{d}{d\lambda}+p_0\Bigg{)}^{d_0} h(\lambda)=0.}
	Then, there are $a_{j,k}\in\C$ such that
	\eqnn{h(\lambda)=\sum_{j=0}^n\sum_{k=0}^{d_j-1} a_{j,k}\lambda^{-p_j}\log^{d_j-1-k}(\lambda).}
	\label{simpprop}
\end{proposition}
\begin{proposition}
	Let $f:(2,\infty)\to\C$ be smooth. Let $0<p< q$ and let $d\in\N.$ If $|(\lddl+p) f(\lambda)|\lesssim \lambda^{-q}\log^{d}(\lambda),$ then $|f(\lambda)|\lesssim \lambda^{-q}\log^{d}(\lambda).$ \label{inductstep}
\end{proposition}

\noindent Proposition \ref{inductstep} provides the base case for the proof of Lemma \ref{diflem}:
\begin{proof}
	Let $D_n$ be the differential operator $(\lddl+p_n)^{d_n}\cdots (\lddl+p_0)^{d_0}.$ Let $h$ be the general solution to the homogeneous equation $D_n(h)=0$ guaranteed by Proposition \ref{simpprop}. Then to solve for $f$ in the differential inequality (\ref{difine}), we need to solve $|D_n(f+h)|\lesssim \lambda^{-p_{n+1}}\log^{d_{n+1}-1}(\lambda).$ We use induction the same way as in the proof of Proposition \ref{simpprop}, using the order $p_0<\cdots <p_n<p_{n+1},$ to conclude
	\eqnn{|f(\lambda)+h(\lambda)|\lesssim \lambda^{-p_{n+1}}\log^{d_{n+1}-1}(\lambda).}
	Hence, there are constants $a_{j,k}\in\C$ such that
	\eqnn{f(\lambda)= \sum_{j=0}^n\sum_{k=0}^{d_j-1} a_{j,k}\lambda^{-p_j}\log^{d_j-1-k}(\lambda) +O\Big(\lambda^{-p_{n+1}}\log^{d_{n+1}-1}(\lambda)\Big).}
\end{proof}
\noindent The conclusion is that for all $n\in\N,$ there are $a_{j,k}\in\C$ such that 
\eqnn{\Bigg{|}I(\lambda)-\sum_{j=0}^n\sum_{k=0}^{d_j-1} a_{j,k}\lambda^{-p_j}\log^{d_j-1-k}(\lambda)\Bigg{|}\lesssim \lambda^{-p_{n+1}}\log^{d_{n+1}-1}(\lambda).}
Finally, taking $p_j$ and $d_j$ as in Theorem \ref{thm2}, the proof is complete, because this inequality implies that any finite sum has error no worse than the next term. Moreover, by standard results involving differentiation under the integral, $I(\lambda)$ is smooth in $\lambda$ no matter the phase considered in this paper.

\subsection{A more general estimate for remainders}
		In the proof of Theorem \ref{thm2}, the integrals
		\eqnn{\int e^{i\lambda\phi(x)} x^\sigma\partial^\sigma R_m(x) x^\beta\psi(x) dx,}
		require a delicate estimation. First, the integral
		\eqnn{\int_{[\varep, 4\varep]} e^{i\lambda\phi(x)} \partial^\sigma R_m(x)x^\beta\eta(x/\varep)dx = \varep^{\beta+\bone}\int_{[1,4]^d} e^{i\lambda\phi(\varep x)} \partial^{\sigma}R_{m}(\varep x)x^\beta\eta(x)dx}
		is considered, for some $C^k$ compactly supported $\eta:[1,4]^d\to \R.$ Write $\partial^\sigma R_k(\varep x)= R_{k,\sigma,\varep}$ and $x^\beta\eta(x)=\eta_\beta(x).$ By \ref{dtest},
		\eqnn{(D^t)^N ( R_{k,\sigma,\varep}\eta_\beta)\lesssim \lambda^{-N} \max_{1\le |\beta^0|\le N}|\partial^{\beta^0} ( R_{k,\sigma,\varep}\eta_\beta)|\varep^{-N\alpha}}
		for all $1\le N\le k-|\sigma|-1$. By the Leibniz formula,
		\alnn{& \partial^{\beta} ( R_{k,\sigma,\varep} \eta_\beta) = \sum_{\alpha_i\le \beta_i} \partial^{\alpha} R_{k,\sigma,\varep}\partial^{\beta-\alpha}\eta_\beta =\sum_{\alpha_i\le \beta_i} \varep^{\alpha}\partial^{\alpha+\sigma} R_{k}(\varep x)\partial^{\beta-\alpha}\eta_\beta\\
			& = \varep^{-\sigma} \sum_{\alpha_i\le \beta_i} \varep^{\alpha+\sigma}\partial^{\alpha+\sigma} R_{k}(\varep x)\partial^{\beta-\alpha}\eta_\beta.\\
			& = \varep^{-\sigma}\sum_{|\gamma|=k} \sum_{\alpha_i\le \beta_i} h_{\alpha,\sigma}(\varep x)\varep^{\gamma}x^\gamma \partial^{\beta-\alpha}\eta_\beta.}
		Since $x\in [1,4]^d,$ for some uniform constant independent of $\varep$ we can bound  
		\eqnn{|\varep^{\alpha+\sigma} \partial^{\alpha+\sigma}R_k(\varep x)|\lesssim\sum_{|\gamma|=k} \varep^{\gamma-\sigma}.}
		Finally,
		\eqn{\Bigg{|}\int_{[\varep, 4\varep]^d} e^{i\lambda\phi(x)}x^\sigma\partial^\sigma R_k(x)x^\beta \eta_\varep(x) dx\Bigg{|}\lesssim \sum_{|\gamma|=k} \lambda^{-N}\varep^{\gamma-N\alpha+\beta+\bone}\label{remestdamn}}
		for all $\alpha\in \cn(\phi)$, all $\varep$ small enough, and all $1\le N\le k-|\sigma|-1$. 

	Now let $\phi\in C^m([-1,1]^d)$ be $k-$nondegenerate and write $\phi=P_{k}+R_{k}$. With the same methods used above, the integral
	\eqnn{I_{R,\sigma}(\lambda)=\int_{[0,4]^d} e^{i\lambda\phi(x)}x^\sigma \partial^\sigma R_k(x) x^\beta\psi(x)dx}
	can be estimated, where $\psi$ is $C^{m-|\sigma|}$ and supported close enough to the origin. To this end, optimization can be applied exactly as above over each $\gamma$, assuming that $m-|\sigma|-1>\lf \gamma+\beta+\bone\rf$ for all $|\gamma|=k$ (so that we can choose $N>\lf \gamma+\beta+\bone\rf$). Since $\lf \gamma+\beta+\bone\rf\ge 1 + \lf\beta+\bone\rf,$ as each $\gamma$ lies inside $\cn(\phi),$ the estimate is
	\alnn{& |I_{R,\sigma}(\lambda)|\lesssim \sum_{|\gamma|=k}\lambda^{-\lf \gamma+\beta+\bone\rf}\log^{d_{\gamma+\beta}-1}(\lambda)\lesssim \lambda^{-1-\lf\beta+\bone\rf}\log^{d'-1}(\lambda),}
	where $d'$ is the greatest codimension over all faces intersecting each line $\{s(\gamma+\beta+\bone):s>0, |\gamma|=k\}.$ In particular, if only convenient phases are considered, $d'=1$ if there is a single compact facet, and in general $d'$ is the greatest codimension over any face of $\cn(\phi)$ not contained in any coordinate hyperplane. This is stated here:
	\begin{lemma}
		Let $\sigma\in\N^d.$ Assume $\phi\in C^m$ is $k-$nondegenerate. Let $\psi$ be $C^{m-|\sigma|}$ with support close enough to the origin and assume $m-|\sigma|-1>\lf\gamma+\beta+\bone\rf$ for all $|\gamma|=k$. Then
		\eqnn{
			\begin{split}
				&\Bigg{|}\int_{\R^d}e^{i\lambda\phi}x^\sigma\partial^{\sigma} R_k x^\beta\psi \Bigg{|}\lesssim \sum_{|\gamma|=k}\lambda^{-\lf \gamma+\beta+\bone\rf}\log^{d'-1}(\lambda)\lesssim \lambda^{-n-\lf\beta+\bone\rf}\log^{d'-1}(\lambda),
			\end{split}}
			where $d'$ is the maximum codimension over any face of $\cn(\phi)$ not contained in coordinate hyperplanes.
	\label{corcor}
\end{lemma}
	 Note that if $\phi$ is smooth, there is no need to worry about $m$.
	
	\section{Derivatives of $I(\lambda)$ and the proof of Theorem \ref{thm2}\label{AsympChapter}}
	\subsection{One derivative: the base case for an induction argument}
	In this section $\ell$ is chosen so that $k<\ell=k(d_0+\cdots+d_n+p_{n+1})<m$ and $\phi$ is expressed as $P_{\ell}+R_{\ell}$ where $P_{\ell}$ is a degree $\ell$ polynomial and $R_{\ell} = \sum_{|\alpha|=\ell}x^\alpha h_\alpha(x).$ Denote the integral $\int_{\R^d}e^{i\lambda\phi(x)}x^\beta\psi(x)dx$ by $I_\beta(\lambda).$
		
	To begin the proof, $P_{\ell}$ is first rewritten in a suggestive way and then  $I_\beta(\lambda)$ is differentiated. First,
	\eqnn{P_{\ell}(x)=\sum_{|\alpha|\le \ell} c_\alpha x^\alpha = \sum_{|\alpha|\le \ell}\sum_{j=1}^d \alpha_j v_j c_\alpha x^\alpha,}
	where we are free to choose any $v\in\R_\ge^d$ satisfying $\alpha \cdot v =1$; the dependence on $\alpha$ is suppressed. Let $w\in \mathbf{n}(\beta+\bone)$ correspond to the facet $F$. Write $P_{\ell}$ as 
	\eqn{\sum_{\alpha}\sum_{j=1}^d \alpha_j (v_j-w_j) c_\alpha x^\alpha+\sum_{\alpha}\sum_{j=1}^d \alpha_j w_j c_\alpha x^\alpha.\label{rightleft1}}
	Note that analytic functions can also be expressed this way. Choosing $v= w$ for all $\alpha\in F,$ the quantity (\ref{rightleft1}) simplifies to 
	\eqn{\sum_{\alpha\notin F}\sum_{j=1}^d \alpha_j (v_j-w_j) c_\alpha x^\alpha+\sum_{\alpha}\sum_{j=1}^d \alpha_j w_j c_\alpha x^\alpha\label{rightleft}.}
	Since the set we are integrating over is compact, and $e^{i\lambda\phi}$ has the same smoothness as $\phi$, the integral $\lddl I_\beta(\lambda)$ equals
	\[\int e^{i\lambda\phi(x)} i\lambda\phi(x)x^\beta\psi(x)dx = 
	\underbrace{\int e^{i\lambda\phi} i\lambda(wx\cdot\nabla\phi)x^\beta\psi}_{I_1} +\underbrace{\int e^{i\lambda\phi} i\lambda(\phi-wx\cdot\nabla\phi)x^\beta\psi}_{I_2}.
	\]
	We first estimate $I_1.$ Integration by parts guarantees
	\eqnn{I_1=- \int e^{i\lambda\phi(x)} \nabla\cdot(x^\beta \psi(x) wx) dx.}
	By the product rule, and because $w\in\bn(\beta+\bone),$ the integral above is just
	\eqnn{I_1= -\lf\beta+\bone\rf \int e^{i\lambda\phi} x^\beta \psi  -\underbrace{\int e^{i\lambda\phi} x^\beta (wx\cdot \nabla\psi)}_{I_{11}}.}
	Define $D_\beta$ to be the operator $\lddl+\lf\beta+\bone\rf$. It was just shown 
	\eqnn{D_\beta I_\beta(\lambda) = I_2-I_{11}.}
	Next, Theorem \ref{wutcor} guarantees
	\eqn{|I_{11}(\lambda)|\lesssim \sum_{j=1}^d w_j \lambda^{-\lf\beta+\bee_j+\bone\rf}\log^{d_j-1}(\lambda).\label{strictsmall}} 
	\begin{itemize}
		\item If $\lf \beta+\bee_j+\bone\rf>\lf\beta+\bone\rf,$ we are done with the estimate. 
		\item Otherwise, $\lf\beta+\bee_j+\bone\rf = \lf\beta+\bone\rf$ so by Proposition \ref{asymprop:p:3} we know $\mathbf{n}(\beta+\bee_j+\bone)\subseteq \mathbf{n}(\beta+\bone).$ 
		\begin{itemize}
			\item[$\circ$] If equality holds, Proposition \ref{asymprop:p:3} guarantees $w_j=0$.
			\item[$\circ$] Otherwise (iv) of the same proposition guarantees the exponent of $\log$ makes the estimate strictly better.
		\end{itemize} 
	\end{itemize}
	To estimate $I_2$, write
	\eqnn{I_2= \underbrace{\int e^{i\lambda\phi} i\lambda(P_{\ell}-wx\cdot\nabla P_{\ell})x^\beta\psi}_{I_{21}} + \underbrace{\int e^{i\lambda\phi} i\lambda(R_{\ell}-wx\cdot\nabla R_{\ell})x^\beta\psi}_{I_{22}}.}
	By (\ref{rightleft}), $P_{\ell}(x)-wx\cdot\nabla P_{\ell}(x)= \sum_{\alpha\notin F}\sum_{j=1}^d (v_j-w_j)c_\alpha x^\alpha.$ If this quantity is zero, we are done with the estimate. Otherwise, Theorem \ref{wutcor} tells us we can bound $I_{21}$ above by
	\eqnn{|I_{21}|\lesssim\max_{\alpha\in \supp(P_{\ell})-F} \lambda \cdot \lambda^{-\lf\alpha+\beta+\bone\rf}\log^{d'-1}(\lambda)}
	for $d'$ guaranteed by Theorem \ref{wutcor}. First, $\alpha\in\cn(\phi)$ implies $\lf\alpha\rf\ge 1$. If $1+\lf\beta+\bone\rf =\lf\alpha+\beta+\bone\rf,$ we again apply Proposition \ref{asymprop:p:3}. Since $\alpha\notin F,$  $\mathbf{n}(\alpha+\beta+\bone)\subsetneq \mathbf{n}(\alpha)\cap \mathbf{n}(\beta+\bone).$ In particular, $|\mathbf{n}(\alpha+\beta+\bone)|<\min\{d, \bn(\beta+\bone)\}.$ Theorem \ref{wutcor} then guarantees that $d'$ must be strictly smaller than the power of $\log$ in the first term of the expansion of $I_{\beta}(\lambda)$, and the estimate in this case is strictly better because of the power of the logarithm. If $\lf\alpha+\beta+\bone\rf>1+\lf\beta+\bone\rf$, the power of $\lambda$ must be strictly smaller than that of the first term in the expansion of $I_\beta(\lambda)$.
	
	We move on to estimating $I_{22}$. The analytic case requires a kind of division argument, and the convenient polyhedron case is done by applying Lemma \ref{corcor}. In the analytic case, we exploit cancellation similar to the treatment of $I_{21}.$ The major obstruction here for $C^m$ functions is that $R_{\ell}(x)-wx\cdot \nabla R_{\ell}(x)$ does not have nice cancellation like we saw for $P_{\ell}(x)-wx\cdot\nabla P_{\ell}(x)$ above if $\phi$ is not smooth, even in one dimension. There is nothing we can do with this term unless $\phi$ is sufficiently smooth, i.e., if $m$ is sufficiently larger than $k.$ 
	
	\subsubsection{The nondegenerate analytic case.}
	Theorem 1.2 in Ziegler\cite{ziegler95} states that any polyhedron is a Minkowski sum of a convex hull of a finite set of points and a conical combination of vectors. In this case, we know $\cn(\phi)$ is the convex hull of its extreme points plus the cone $\R_\ge^d.$ Because of this fact, assume $m$ is so large that $\alpha\in \cn(\phi)$ with $|\alpha|\ge m$ implies $\alpha=\alpha'+\bee_j$ for some $\alpha'\in \cn(\phi)$ and some $1\le j\le d.$ Then we can show, using the same cancellation above for $P_m(x)-wx\cdot\nabla P_m(x)$, that $R_m(x)-wx\cdot\nabla R_m(x)$ can be written as
		\eqn{\sum_{\substack{|\alpha| \ge m\\ \alpha\notin H_w}} c_\alpha' x^\alpha = \sum_{\substack{|\alpha| = m\\ \alpha\notin H_w}} x^\alpha h_\alpha(x)\label{remainder}}
	for some analytic $h_\alpha(x).$ This is by induction: clearly any $|\alpha|=m$ can be expressed as $\alpha+\bz$ where $\alpha\notin H_w.$ For the induction step, we simply use that any $\beta$ can be expressed as $\beta'+\bee_j$, where $\beta'=\alpha+\gamma$ and $|\alpha|=m, \alpha\notin H_w.$ Therefore all $|\alpha|\ge m$ fall into equivalence classes depending on which monomial $x^\alpha$ of order $m$ divides them, with $\alpha\notin H_w.$ With this we can complete the estimate for analytic phases with the exact same method as $I_{21}$: after taking the remainders to be as discussed above in this paragraph, by Lemma \ref{corcor} with $h_\alpha\psi$ as the smooth amplitude,
	\eqnn{|I_{22}| \lesssim \lambda^{1- \lf \alpha+\beta+\bone\rf}\log^{d'-1}(\lambda).}
	To make sure this bound is better than what we started with, apply the exact same argument as for $I_{21}$ for the finitely many $|\alpha|=\ell$ not contained in $F=H_w\cap \cn(\phi).$

	\subsubsection{The $k-$nondegenerate case.}	
	To finish this case, we use that $\ell=k(d_0+\cdots+d_n+p_n).$ For any $|\gamma|=\ell,$ let $|\alpha|=k$ be such that $(d_0+\cdots+d_n+p_n)\alpha=\gamma.$ Using this, together with convenience of $\phi,$ we bound $\lf\gamma\rf$ above and below. By Proposition \ref{asymprop:p:1}, $\lf \alpha\rf \ge 1$ implies $\lf \gamma\rf \ge d_0+\cdots+d_n+p_n$. Since $\cn(\phi)$ intersects each coordinate axis, all normals $w$ corresponding to facets of $\cn(\phi)$ have positive components. On the other hand, all normals $w$ corresponding to facets satisfy $\bee_j\cdot w\le 1$ because all such supporting hyperplanes lie above the plane $|\xi|=1.$ Therefore for any $w\in\bn(\gamma+\bone)$ we have $\lf\gamma+\bone\rf = (\gamma+\bone)\cdot w \le \ell+d.$ The inequality is because $|\gamma+\bone|=\ell+d.$ We summarize this for future use:
	\eqn{d_0+\cdots+d_n+p_n\le \lf\gamma+\bone\rf \le \ell+d.\label{rembounds}}
	Lemma \ref{corcor} guarantees that if $m>d_0+\cdots +d_n+\lf\alpha+\bone\rf$ then
	\eqnn{|I_{22}| \lesssim\max_{|v|=\ell} \lambda^{1- \lf v+\beta+\bone\rf}\log^{d'-1}(\lambda).}
	Therefore we can apply the estimate for 
	\eqnn{m>d_0+\cdots +d_n+\ell+d=(k+1)(d_0+\cdots+d_n)+kp_n+d\hspace{.25cm}\footnote{This is just a technical condition to significantly reduce the potential complexity of computations and bookkeeping. Lemma \ref{corcor} can be iterated in order to examine what happens for products of remainders. By keeping track of where each derivative lands, one could be less restrictive on the lower bound on $m$ using these same arguments.}}
	and the estimate is $\lesssim \lambda^{1-\lf \gamma+\beta+\bone\rf}\log^{d'-1}(\lambda)\lesssim \lambda^{-\lf\beta+\bone\rf-p_n}\log^{d'-1}(\lambda).$
	
	Let $p_0<p_1<\cdots$ be the well ordering stated at the beginning of the section. Let $j\in\N$ be such that $p_j=\lf\beta+\bone\rf.$ To summarize, the above shows
	\begin{eqnarray}
	\displaystyle|D_\beta I_\beta(\lambda)|\lesssim
	\begin{cases}
	\displaystyle \lambda^{-p_{j+1}}\log^{d_{j+1}-1}(\lambda)& d_j=1\\
	\displaystyle \lambda^{- p_j}\log^{d_j-2}(\lambda)  & \text{otherwise.}\\
	\end{cases}\label{asympindbase}
	\end{eqnarray}
	
	\subsection{Estimating higher derivatives of $I(\lambda)$}
	Since the induction is the same for all $d_j, r,$ except for a few extra difficulties when $r=d_n,$ we consider only the following case. For $0\le r<d_n$, denote by $G_{n,r}(\lambda)$ the integral
		\eqnn{\Bigg{(}\lddl+p_n\Bigg{)}^r\Bigg{(}\lddl+p_{n-1}\Bigg{)}^{d_{n-1}}\dotsm \Bigg{(}\lddl+p_0\Bigg{)}^{d_0}I(\lambda).}
	The goal is to show
		\begin{eqnarray}
		\displaystyle|G_{n,r}(\lambda)|\lesssim
		\begin{cases}
		\displaystyle \lambda^{-p_{n+1}}\log^{d_{n+1}-1}(\lambda)& d_n=1\\
		\displaystyle \lambda^{- p_n}\log^{d_n-2}(\lambda)  & \text{otherwise. }\\
		\end{cases}\label{difqq}
		\end{eqnarray}
	The induction hypothesis is that $G_{n,r}$ can be expressed as a sum 
		\eqnn{G_{n,r}(\lambda)=\sum_{j=0}^{d_0+\cdots+d_{n-1}+r} \lambda^j J_{j,n,r}(\lambda),}
	where in the $k-$nondegenerate case, $J_{j,n,r}$ can be split up into finitely many integrals, each of which are either of the form
	\eqnn{I_1=\int e^{i\lambda\phi(x)} x^\alpha \psi(x)dx,\text{ or the form }I_2=\int e^{i\lambda\phi(x)} x^\sigma\partial^{\sigma}R_{\ell}(x)\psi(x)dx,}
	where
	\begin{itemize}
		\item each $\sigma$ satisfies $|\sigma|\le d_0+\cdots+d_{n-1}+r$;
		\item each $\alpha$ satisfies $p_n+j\le \lf\alpha+\bone\rf,$ with equality only if the line $s(\alpha+\bone)$ does not intersect $\cn(\phi)$ in a codimension $d_n-r+1$ face; 
		\item each $\psi$ are compactly supported and at least $m-d_0-\cdots-d_{n-1}-r$ times continuously differentiable.
	\end{itemize}
	In the real analytic case, the induction hypothesis is reduced to only integrals of the form $I_1$, under only the second and third hypotheses above.
	
	The base case $(n=0,r=1)$ was completed above (choosing $\beta=\bz$). Applying $(\lddl+p_n)$ to $\lambda^j J_{j,n,r}$ consider $(\lddl+p_n)\lambda^jI_1:$ after a similar computation to what was already done, this expression equals
	\aln{(p_n+j-(\alpha+\bone)\cdot w)\lambda^j&\int e^{i\lambda\phi} x^\alpha\psi\label{I11}\\
		- \hspace{.2cm}\lambda^j&\int e^{i\lambda\phi}x^\alpha (wx\cdot\nabla\psi)\label{I12}\\
		 + \hspace{.2cm}i\lambda^{j+1}&\int e^{i\lambda\phi} x^\alpha(P_\ell-wx\cdot \nabla P_\ell)\psi\label{I13}\\
		 + \hspace{.2cm}i\lambda^{j+1}&\int e^{i\lambda\phi} x^\alpha(R_\ell-wx\cdot \nabla R_\ell)\psi\label{I14}.}
	We can make a choice of any vector $w$ above, so choose $w\in\bn(\alpha+\bone).$ 
	
	If $\lf\alpha+\bone\rf >p_n+j,$ the estimate is strictly better by Lemma \ref{corcor} and we are done. Otherwise, $\lf\alpha+\bone\rf=p_n+j$ and the coefficient in (\ref{I11}) equals zero. Integral $(\ref{I11})$ satisfies the induction hypothesis for the next derivative. 
	
	Lemma \ref{corcor} bounds (\ref{I12}) above by $\lambda^j\sum_{i=1}^d w_i \lambda^{-\lf\alpha+\bee_{i}+\bone\rf}\log^{d_\alpha-1}(\lambda)$. Hence, for $k-$nondegenerate phases, we are done: $\lf\alpha+\bee_i+\bone\rf> \lf\alpha+\bone\rf\ge p_n+j.$ In the analytic case, if $w\in \bn(\alpha+\bee_i+\bone)$ and $w'\in \bn(\alpha+\bone),$ then
	\eqnn{(\alpha+\bee_i+\bone)\cdot w=(\alpha+\bone)\cdot w'\le (\alpha+\bone)\cdot w}
	so that $w_i=0$ for all $w\in\bn(\alpha+\bee_i+\bone).$ Therefore the coefficient of $\lambda^{-\lf\alpha+\bee_{i}+\bone\rf}\log^{d_\alpha-1}(\lambda)$ in this estimate is zero. Therefore the induction hypothesis is satisfied here.
	
	For (\ref{I13}), the argument is identical as the one given for integral $I_{21}$ in the induction step. It is also easy to see this integral satisfies the induction step.
	
	Integral (\ref{I14}) satisfies the induction hypothesis after we break up the difference of the remainders into two integrals for the $C^k$ case. The estimate is an identical argument to that for $I_{22}$ above in both cases, which must be considered separately in order to use the previous methods. In this case Lemma \ref{corcor} provides the estimate 
	\eqnn{(\ref{I14})\lesssim \max_{|\gamma|=\ell}\lambda^{j+1-\lf\alpha+\gamma+\bone\rf}(\lambda)\log^{d-1}(\lambda).}
	But $d_0\ge 1$, $p_{n+1}>0$, and Proposition \ref{asymprop:p:2} imply
	\eqnn{\lf\alpha+\gamma+\bone\rf \ge \lf\alpha+\bone\rf +\lf\gamma\rf \ge p_n+j+(d_0+\cdots+d_n+p_{n+1})>p_n+j+1.}
	
	Now we consider integrals of the second form for $k-$nondegenerate phases. Computing $(\lddl+p_n)I_2$, we get summands
	
	\alnn{(p_n+j-(\sigma+\bone)\cdot w)\lambda^j&\int e^{i\lambda\phi} (x^\sigma\partial^\sigma R_\ell)\psi
		- \lambda^j\int e^{i\lambda\phi}(x^\sigma\partial^\sigma R_\ell) (wx\cdot\nabla\psi)\\
		- \hspace{.2cm}\sum_{a=1}^d w_a\lambda^j&\int e^{i\lambda\phi}(x^{\sigma+\bee_j}\partial^{\sigma+\bee_j} R_\ell) (wx\cdot\nabla\psi)\\
		+ \hspace{.2cm} i\lambda^{j+1}&\int e^{i\lambda\phi} (x^\sigma\partial^\sigma R_\ell)(P_\ell-wx\cdot \nabla P_\ell)\psi\\
		+ \hspace{.2cm}i\lambda^{j+1}&\int e^{i\lambda\phi} (x^\sigma\partial^\sigma R_\ell)(R_\ell-wx\cdot \nabla R_\ell)\psi.
		}
	By choice of degree of the remainder terms, the decay of all of these integrals is strictly better for the $k-$nondegenerate case. Moreover, we easily see all of these integrals either are equal to a power of $\lambda$ times an integral of the form $I_2,$ or can be expressed as a sum of such. This completes the argument.

	\section{Acknowledgments}
	I want to deeply thank the many people who read and discussed early drafts of this paper. Most importantly I would like to thank my thesis adviser, Professor Philip T. Gressman, for everything he has taught me. Without him surely this would have been impossible for me.
	
	

	\bibliographystyle{plain}
	\bibliography{Article1}
	
\end{document}